\documentclass[11pt]{article}
\usepackage{amsfonts}
\usepackage{amssymb}
\usepackage{amsthm}
\usepackage{amsmath}
\usepackage{fixltx2e}
\usepackage{gensymb}
\usepackage[numbers]{natbib}

\usepackage{indentfirst}
\usepackage{mathrsfs}
\usepackage{tikz}
\usepackage{braket}
\theoremstyle{plain}
\usepackage{enumerate}
\usepackage{upgreek}
\usepackage{enumitem}
\usepackage{siunitx}
\newtheorem{theorem}{\textbf{Theorem}}[section]
\newtheorem{lemma}{\textbf{Lemma}}[section]
\newtheorem{proposition}{\textbf{Proposition}}[section]

\newtheorem{remark}{\textbf{Remark}}[section]
\newtheorem{definition}{\textbf{Definition}}[section]

\allowdisplaybreaks[4]

\newcommand{\bt}[0]{\bar\theta}
\newcommand{\bart}[0]{\bar{\theta}}
\newcommand{\te}[0]{\theta  }
\newcommand{\Ltv}[0]{L^2(0,T;H^1_0)}
\newcommand{\dep}[0]{\delta^\eps_{u(t)}}

\def\R3{\mathbb{R}^3} 
\def\R{\mathbb{R}}
\def\F2o{\overline{F_2}}

\def \E{\mathcal E}

\def \eps{\varepsilon}

\newcommand{\beq}[0]{\begin{equation}}
\newcommand{\eeq}[0]{\end{equation}}

\newcommand{\mint}[0]{\int_\Omega}
\newcommand{\tint}[0]{\int_0^t}
\newcommand{\Tint}[0]{\int_0^T}

\newcommand{\dert}[0]{\frac{d}{dt}}

\newcommand{\EE}[0]{{\mathcal E}}

\newcommand{\vc}[1]{\mathbf{#1}}
\newcommand{\mt}[1]{\mathsf{#1}}
\newcommand{\Zn}[0]{Zn\textsubscript{45}Au\textsubscript{30}Cu\textsubscript{25}}

\newcommand{\meas}[0]{\mathscr{L}^3}

\DeclareMathOperator{\supp}{supp}
\DeclareMathOperator{\tr}{tr}

\DeclareMathOperator{\argmin}{argmin}

\DeclareMathOperator{\Div}{Div}

\DeclareMathOperator*{\esssup}{ess\,sup}

\newcommand{\tow}[0]{\rightharpoonup}

\usepackage[hang,flushmargin]{footmisc}
\makeatletter
\def\blfootnote{\gdef\@thefnmark{}\@footnotetext}
\makeatother

\begin{document}

\title{A model for the evolution of highly reversible martensitic transformations}

\author{
{\sc Francesco Della Porta}\\
Mathematical Institute, University of Oxford\\
Oxford OX2 6GG, UK\\
\textit{francesco.dellaporta@maths.ox.ac.uk\footnote{Current institution: Max-Planck Institute for Mathematics in the Sciences, Leipzig, Germany, 
\textit{Francesco.DellaPorta@mis.mpg.de}}}
}

\date{\today}
\maketitle

\begin{abstract}
In this work we introduce a new system of partial differential equations as a simplified model for the evolution of reversible martensitic transformations under thermal cycling in low hysteresis alloys. The model is developed in the context of nonlinear continuum mechanics, where the developed theory is mostly static, and cannot capture the influence of dynamics on martensitic microstructures. First, we prove existence of weak solutions; secondly, we study the physically relevant limit when the interface energy density vanishes, and the elastic constants tend to infinity. The limit problem provides a framework for the moving mask approximation recently introduced by the author. In the last section we study the limit equations in a one-dimensional setting. After closing the equations with a constitutive relation between the phase interface velocity and the temperature of the one-dimensional sample, the equations become a two-phase Stefan problem with a kinetic condition at the free boundary. Under some further assumptions, we show that the phase interface reaches the domain boundary in finite time. 
\end{abstract}


\begin{section}{Introduction}
Martensitic transformations are solid-to-solid phase transitions, an abrupt change in the crystalline structure occurring in certain alloys or ceramics when their temperature crosses a critical value $\theta_T$. The high-temperature phase is usually called austenite, the low-temperature one martensite. These transformations are important, for example, in shape-memory alloys, which are materials able to recover, upon heating, deformations that are apparently plastic. Shape memory alloys have a wide range of applications, from medical stents to smart actuators. However, damages arising from multiple thermal cycles, deteriorating reversibility of the transformation, are still major obstacles to their use.\\
A previous study, identified some conditions of geometric compatibility between austenite and martensite, called cofactor conditions, and speculated they might affect reversibility \cite{JamesHyst}. This conjecture was partially confirmed with the discovery of Zn\textsubscript{45}Au\textsubscript{30}Cu\textsubscript{25} \cite{JamesNew}. This is the first material closely satisfying the cofactor conditions, with a relative error of order $10^{-4}$, which shows no loss of reversibility of the transformation after more than 16000 thermal cycles. This result is quite surprising if compared with other materials, where nucleation of micro-crack deteriorates reversibility, and rupture occurs after only a few thousands cycles (see e.g., \cite{Chluba}). However, a rigorous mathematical understanding of reversibility and its connections with the cofactor conditions is still missing.\\
An important step towards understanding reversibility, is to fully characterise the martensitic microstructures arising during the transformation. These have been successfully studied in the context of nonlinear continuum mechanics, where martensitic transformations are treated as elastic deformations \cite{BallJames1,BallJames2,Batt}. This theory relies on minimising an energy functional in order to predict the local deformations of the crystalline structure. As proven in \cite{MullerSverak}, we often find ourselves in the presence of infinitely many minimisers, with the drawback of not being able to discriminate which ones are physically relevant.\\
In a previous study, in order to address this problem, based on experimental observations with Zn\textsubscript{45}Au\textsubscript{30}Cu\textsubscript{25}, we introduced the moving mask approximation \cite{FDP2}. This, under suitable assumptions, enabled characterisation of the macroscopic deformation gradients for martensite as the ones of the form
$$
\mt 1 + \vc a(\vc x) \otimes \vc n(\vc x), \quad\text{a.e. $\vc x$}, 
$$
for some $\vc a, \vc n \in L^\infty(\Omega; \R^3)$. This result, which would not have been achievable in a static framework, turned out to be in accordance with experiments on Zn\textsubscript{45}Au\textsubscript{30}Cu\textsubscript{25} \cite{XC,FDP2,JamesNew}, and appears relevant for many other alloys. \\
In order to frame the moving mask approximation in a dynamical model in the context of continuum mechanics, with the aim of better understanding the complex microstructures arising in Zn\textsubscript{45}Au\textsubscript{30}Cu\textsubscript{25}, in this paper we introduce a simplified model to describe the evolution of thermally induced martensitic phase transitions with ultra-low thermal hysteresis. {Here and below, by thermal hysteresis (or simply hysteresis) we mean $\frac12(A_s+A_f-M_s-M_f)$, $A_s,A_f$ (resp. $M_s,M_f$) being the temperatures where austenite (resp. martensite) starts to appear and ends to appear during thermally induced martensitic transformations. This is consistent with the way hysteresis is measured in the literature when the phase transformation is induced by locating the sample on a thermostat and by varying its temperature (see e.g., \cite{JamesHyst,JamesMuller}).} Low hysteresis hence means that the transformation takes place in a small range of temperature values. This is the case, for example, for Zn\textsubscript{45}Au\textsubscript{30}Cu\textsubscript{25} where the experimental values of hysteresis are of $2\,^\circ \mathrm K$, in contrast to values up to $70\,^\circ \mathrm K$ in NiTi alloys. Our model results in a system of partial differential equations, for which we prove existence of weak solutions, and that we relate to the moving mask approximation of \cite{FDP2}.
\\
The structure of this paper is the following: in Section \ref{nonlonear elastic}, we introduce the nonlinear elasticity model to describe martensitic phase transitions. In the same section, we recall the equations describing the balance of momentum and the balance of energy in continuum mechanics. Indeed, following the work in \cite{AberKnowles,AberKnowles2}, it seems relevant for the evolution of the austenite-martensite interface to take into account thermodynamic effects. In Section \ref{potential}, we introduce the assumptions on the free-energy density of the system, and the boundary conditions, which reflect the experimental setup for \Zn \,(see \cite{JamesNew}). Furthermore, we justify the linearization of the equation for the conservation of energy around the critical temperature. Existence of solutions to the resulting system of partial differential equations is proved in Section \ref{existence full}. This proof requires the application of a version of the div-curl lemma (see e.g., \cite{All}) to show that the solutions to a suitably defined approximated problem converge to weak solutions of our problem, when the regularisation vanishes. In Section \ref{inelali} we study the limit of the solutions as the elastic constants tend to infinity and the interface energy density goes to zero. This approximation is physically relevant and has been already adopted, for example in \cite{BallChuJames} and \cite{BallKoumatosQC}. Here we prove that the deformation gradients $\nabla\vc y$ generate in the limit a gradient Young measure $\nu_{\vc x,t}$ which is supported on the phases, and whose evolution cannot be deduced from the conservation of momentum alone. 
%
On the other hand, the equation for the conservation of energy becomes a heat equation with a source term of the type $\dert \int_{\R^{3\times3}}\eta_1(\mt A)\,\mathrm{d}\nu_{\vc x,t}(\mt A)$, {where $\eta_1-\log\Bigl(\frac{\theta}{\theta_T}\Bigr)$ is the entropy of the system, and $\eta_1$ is supposed to be a smooth function, assuming different values in austenite and martensite}. In Section \ref{connection} we recall the moving mask approximation introduced in \cite{FDP2}, and explain how it can be framed in the context of our limit problem. In this section, we also devise a possible strategy to make the connection rigorous.\\
In the last section, we study the propagation of a laminate in a long bar, where we can simplify the problem, under some further assumptions, to a one-dimensional one. The resulting equation is an easier version of the equations introduced in \cite{AberKnowles2}, that are underdetermined, and should be closed with a constitutive relation for the velocity of the phase boundary. 
After introducing a constitutive relation between the velocity and the temperature of the phase boundary, our problem becomes a two-phase Stefan problem with kinetic condition at the free boundary (see e.g., \cite{Vis,VisLib}). We prove that, under our assumptions, the position of the austenite-martensite phase interface is a monotone function of time, which reaches the domain boundary in finite time. This simple, one-dimensional case is an example of solution satisfying the moving mask approximation to the limit problem obtained in Section \ref{inelali}. \\
\end{section}
\section{Preliminaries}

\label{nonlonear elastic}
\subsection{Nonlinear elasticity model}
One way to model martensitic phase transitions at a continuum level is in the context of nonlinear elasticity(see e.g., \cite{BallJames1} and \cite{Batt}). In this framework, one can look at phase transformation from austenite to martensite as elastic deformations minimizing the given free energy
\beq
\label{energia}
\EE(\vc y,\theta)  = \mint\phi(\nabla \vc y(\vc x),\theta)\,\mathrm d\vc x.
\eeq
Here and below the bounded smooth domain $\Omega\subset\R^3$ represents a sample in its undeformed austenite phase at the critical temperature, while $\vc y(\vc x)\in\R^3$ is the position in space occupied by the point of material $\vc x\in\Omega$ after the transformation. The temperature of the body is represented by $\theta$, while $\phi$ is the free-energy density that reflects the fact that below the critical temperature $\theta_T$ martensite is energetically preferable to austenite, while the opposite holds when $\theta>\theta_T$. Defined $\mathcal D:=\{\mt F\in\mathbb{R}^{3\times 3}\colon \det \mt F >0\}$, standard assumptions on $\phi(\cdot,\theta)\colon\mathcal D\to \mathbb{R}$ are (see e.g., \cite{BallKoumatos,FDP2})
\begin{itemize}
\item  $\phi(\cdot,\theta)$ is a function bounded below by a constant depending on $\theta$ for each $\theta>0$;
\item $\phi(\cdot,\theta)$ satisfies frame-indifference, i.e., for all $\mt F\in\mathcal D$ and all rotations $\mt R\in SO(3)$, $\phi(\mt R \mt F,\theta)=\phi(\mt F,\theta)$. This property reflects the invariance of the free-energy density under rotations;
\item $\phi$ has cubic symmetry, i.e., for all $\mt F\in\mathcal D$ and all rotations $\mt Q$ in the symmetry group of austenite $\mathcal{P}_a$, $\phi(\mt {FQ},\theta)=\phi(\mt F,\theta)$ must hold (see \cite{Batt} for more details);
\item denoting by $K_\theta$ the set of minima for the free-energy density at temperature $\theta$, i.e., $K_\theta:=\{\mt F\in\mathcal D\colon\,\mt F\in\argmin(\phi(G,\theta))\}$,
\beq
K_\theta = 
\begin{cases}
\alpha(\theta)SO(3),\qquad &\theta>\theta_T\\ 
SO(3)\cup\bigcup_{i=1}^NSO(3)\mt U_i(\theta_T), \qquad &\theta=\theta_T\\
\bigcup_{i=1}^NSO(3)\mt U_i(\theta), \qquad &\theta<\theta_T .
\end{cases}
\eeq
Here, $\alpha(\theta)$ is a scalar dilatation coefficient satisfying $\alpha(\theta_T)=1$, while {
$\mt U_i(\theta)\in\R^{3\times3}_{Sym^+}$ are the $N$ positive definite symmetric matrices corresponding to the transformation from austenite to the $N$ variants of martensite at temperature $\theta$.} From now on, we omit the dependence on the temperature in $K_\theta$ when $\theta<\theta_T$, and we neglect the dependence of $\alpha$ and the $\mt U_i$ from $\theta$. {In \Zn, for example, the transformation is from a cubic to a monoclinic II lattice, and the $N=12$ variants of martensite can be found in \cite[Table 4.4]{Batt}. Our setup is however very general, and the analysis below does not rely on $N$ and on the shape of the $\mt U_i$'s. }
\end{itemize}
At the critical temperature, austenite and martensite are energetically equivalent; therefore the energy $\E$ is not quasiconvex and minimisers attaining both phases might therefore not exist. {For this reason, in the literature authors have often considered a physically relevant regularisation of $\E$ which takes into account interface energy. This is usually done by penalizing the second derivatives in the $BV$ or in the $L^2$ norm, and has shown to characterise physically relevant minimisers, and to capture many aspects of the physical phenomenon such as the finite scale of the microstructures and the branching of laminates (see e.g., \cite{BallJames1,ContiSchweizer,FDP4,MullerDolzman,KohnMuller,JamesMuller}).  } However, even if this model is more accurate, it is more difficult to capture analytically the behaviour of solutions.\\ 

On the other hand, by arguing as in \cite{BallJames1,BallJames2} it is often possible to construct minimising sequences which allow coexistence of both austenite and martensite. A way to deal with the lack of weak lower-semicontinuity of the functional is thus to consider the relaxed problem
$$
\E^r(\vc y,\nu_{\vc x}) = \mint \int_{\R^{3\times3}}\phi(\mt F,\theta_T) \,\mathrm{d}\nu_{\vc x}\,\mathrm{d}\vc x,
$$
where $\nu_{\vc x}$ is a gradient Young measure with barycentre $\nabla \vc y$ for some $\vc y\in W^{1,\infty}(\Omega;\R^3)$ (see e.g., \cite{BallJames2, Muller} and also \cite{Pedregal} regarding gradient Young measures). In this way, under mild assumptions on $\phi$ the problem always admits minimisers, and the gradient Young measure can keep track of the oscillations in the minimising sequences. Thus, $\nu_{\vc x}$ embeds the information about the microstructures, while the barycentre of $\nu_{\vc x}$, namely $\nabla\vc y (\vc x):=\int_{\R^{3\times3}}\mt F \,\mathrm{d}\nu_{\vc x}(\mt F)$, is also called macroscopic deformation gradient, because it is an average of the fine microstructures. We have $\nabla\vc y \in K^{qc}$ a.e. in $\Omega$, where
$$
K^{qc}:=\bigl\{\mt M\in \R^{3\times 3}\colon f(\mt M)\leq \max_K f,\,\text{for all quasiconvex } f\colon\R^{3\times3}\to\R \bigr\},
$$
is the quasiconvex hull of the set $K$ (see \cite{Muller}).\\

Characterising the quasiconvex hull of a generic set is an open problem which strongly affects our understanding of martensitic transformations. Nonetheless, the nonlinear elasticity model has been successfully used to understand many aspects of this phenomenon, such as twinning (see \cite{BallJames1}), shape-memory effect (see \cite{BattSme}), and, more recently, hysteresis
(see \cite{JamesMuller}).

\subsection{Dynamics of martensitic phase transitions in continuum solid mechanics}
{
Existing studies in the literature have studied the evolution of martensitic transformations at a continuum level within two different frameworks: the first one is the framework of phase field models, the second one is the context of the dynamics for nonlinear continuum mechanics. In the former case, the system is described by one (or more) order parameter representing the volume fraction of a phase (or of a martensitic variant) at each (or at almost each) point of the domain. Examples can be found in \cite{PF3,PF2,PF1} and references therein. These models have been used in the literature to capture aspects of the phenomenon like fatigue and hysteresis. However, to our knowledge, there is no analytical result stating that the compatibility between crystalline phases is preserved during the whole transformation. Furthermore, no result proves that for every (or almost every) time instant the deformation map $\vc y$, describing the change of atoms positions during the lattice transformation, is continuous. This is observed in experiments and should be reflected by the model. In the context of nonlinear elasticity, the continuity of the deformation map $\vc y$ is assured by Sobolev embeddings, and the crystal compatibility holds in some weak form on every plane in the domain (see e.g., \cite[Remark 5.1]{FDP2}). Furthermore, compatibility between crystalline phases, and continuity of the underlying crystal deformation map have been proven to be important tools to understand microstructures (see e.g., \cite{BallJames1,Batt}). For these reasons, in this manuscript we decided to focus on the nonlinear elasticity model and work within this framework, which seems better suited to describe the complex microstructures arising in \Zn.
%
%
%
}

As mentioned in the previous section, the statics of martensite to austenite phase transitions 
can be successfully modelled in the framework of nonlinear elasticity. 
For these reasons, many authors have tried to study the evolution of martensitic phase transitions within the same theory, where the conservation of linear momentum can be expressed in Lagrangian coordinates as 
\beq
\label{evolution 0}
\rho_0 \ddot {\vc y }= \Div \upsigma.
\eeq
Again, here $\vc y(\vc x,t)$ is the position at time $t$ of the material point $\vc x\in\Omega$, $\rho_0(\vc x)$ the density of the body in its reference configuration, and $\mt\sigma$ the Piola-Kirchhoff stress tensor, whose dependence on $\vc y$ and the deformation gradient $\mt F=\nabla \vc y$ is given in terms of the free energy $\phi$ of the material through the relation 
\beq
\label{def sigma}
\upsigma(\mt F,\theta)=\frac{\partial\phi}{\partial \mt F}(\mt F,\theta).
\eeq
However, as $\phi$ is in general not rank-one convex, the operator $\Div (\upsigma(\nabla \vc y,\theta))$ is not elliptic, making the problem extremely complex from a mathematical point of view, and the related initial value problem formally ill posed. 
Furthermore, the total energy is conserved in the model for smooth solutions, which is not coherent with physical observations, where heat is absorbed or released during the process. Therefore, one should seek solutions with shocks that, as explained for example in \cite{AberKnowles}, dissipate energy, but which are mathematically very difficult to study. 
Also studied in the literature are viscoelastic models. In the case of viscoelasticity of rate type the stress tensor has the form
$$
\upsigma = \upsigma(\nabla \vc y,\nabla\dot{\vc  y},\theta),
$$ 
in \eqref{evolution 0}, or in its quasi-static variant where $\ddot{\vc y}$ is assumed to be negligible. 
However, many difficulties arise in dimension greater than one due to the fact that frame-indifference prohibits a linear dependence on $\nabla\dot{\vc y}$
(see e.g., \cite{Sengul_thesis}). For further discussion, we refer the reader to \cite{BallOpen, BallPego, BallSengul, Pego, Sengul_thesis} and references therein.
\\

On the other hand, the phase transition generates heat, and it is therefore reasonable, when modelling the evolution for the austenite to martensite phase transition, to take into account thermal effects. The equations of continuum mechanics expressing the balance of linear momentum and energy in the absence of heat supply and neglecting body forces are given by (see e.g., \cite{BallOpen,Truesdell})
\begin{align*}
\frac\partial{\partial t}\Bigl(\rho_0\frac12 |\dot {\vc y}|^2+U\Bigr)-\Div(\mt C(\nabla \vc y,\theta)\nabla\theta) - \Div(\dot{\vc y}\upsigma(\nabla \vc y,\theta)) &=0,\\
\rho_0\ddot{\vc y}-\Div \upsigma(\nabla \vc y,\theta)& = \vc 0.
\end{align*}
Here, $U$ is the internal energy of the system, and we also made use of the Fourier law for the heat flux $\vc q$, that is $\vc q=-\mt C\nabla \theta$. For simplicity $\mt C$ will be considered to be a constant, symmetric, positive definite matrix in $\R^{3\times3}$ with smallest eigenvalue equal to $c_m$ and $\rho_0$ is assumed constant. 
{The term $\Div(\dot{\vc y}\upsigma(\nabla \vc y,\theta))$ in the energy balance accounts for the energy dissipated or generated by internal forces and, by using the equation for the conservation of momentum, can be rewritten as $\rho_0\frac12 \frac\partial{\partial t}|\dot {\vc y}|^2 + \nabla\dot{\vc y}:\upsigma(\nabla \vc y,\theta) $.}
Thus, after exploiting the relation $U=\phi+\theta\eta$ relating the internal energy to the entropy $\eta$ and the free energy $\phi$, the above equations become
\beq
\label{Pb 0}
\begin{split}
\rho_0 \theta \dot\eta(\nabla \vc y,\theta)-\Div(\mt C\nabla\theta) &=0,\\
\rho_0\ddot{\vc y}-\Div \upsigma(\nabla\vc y,\theta)&=\vc 0,
\end{split}
\eeq
where we also made use of
\beq
\label{entropy def 0}
\eta(\nabla\vc  y,\theta)=-\frac{\partial\phi}{\partial \theta}(\nabla \vc y,\theta).\eeq


\section{Hypotheses, boundary conditions and approximations}
\label{potential}
{
The aim of this section is to introduce some hypotheses on the free-energy density $\phi$. Given the fact that we are restricting ourselves to elastic deformations, and that elastic moduli are in general considerable in metals, minimisers of the free energy lie extremely close to the wells. Therefore, it is a reasonable approximation to adopt in Section \ref{inelali} the approach of \cite{BallKoumatosQC}, and let $\phi$ grow to infinity outside the wells, thus neglecting the influence of the shape of $\phi(\mt F)$ for $\mt F\notin K\cup SO(3)$ on our results.\\}
{
Following the approach of \cite{AberKim} and references therein, we split the energy contributions into a term, $\phi_3$ below, responsible for the increase of energy with temperature, and terms, namely $\phi_1,\phi_2$, describing the energy thermo-mechanical effects. 
As hysteresis is smaller than $5\,^\circ \mathrm K$ in materials undergoing ultra-reversible martensitic transformations such as Zn\textsubscript{45}Au\textsubscript{30}Cu\textsubscript{25}, we can neglect the effects of thermal expansion, and we also neglect the heat generated by elastic deformations which are not a change of phase, which, as mentioned above, are usually observed to be small. For $k\geq 1$ we thus consider
\beq
\label{energia kk}
\phi^k(\mt F,\theta) = \phi_1(\mt F,\theta)+{k}\phi_2(\mt F)+\phi_3(\theta),
\eeq
where:
\begin{itemize}
\item $\phi_2\colon \R^{3\times3}\to [0,\infty)$ is smooth and satisfies
$$
\phi_2(\mt F) = 0 \Leftrightarrow \mt F\in K\cup SO(3),\qquad c_2(|\mt F|^{
p}+1)\geq\phi_2(\mt F)\geq c_0|\mt F|^p-c_1,
$$
with $c_0,c_1,c_2>0$ and $p\in(2,\infty)$. 
When $k\gg 1$, the $\phi_2$ term is responsible for the growth of the energy outside the wells, and $k$ plays the role of an elastic constant. {An example of $\phi_2$ in the case of $N\geq 2$ martensitic variants $\mt U_i$ is given by $\phi_2(\mt F) = |\mt 1 - \mt F^T\mt F|^2\prod_{i=1}^{N}|\mt U_i^T\mt U_i - \mt F^T\mt F|^2$};
{\item 
$
\phi_3(\theta)=\gamma\theta\Bigl(1-\log\Bigl(\frac{\theta}{\theta_T}\Bigr)\Bigr),
$
arises naturally by assuming that the specific heat $\gamma>0$ is constant, and hence independent of the phase and of the temperature. Indeed, $\phi^k$ must satisfy $\gamma = -\theta\frac{\partial^2 \phi^k}{\partial\theta^2}$ (see \cite{AberKnowles2});}
\item
$\rho_0\phi_1 = \bar\phi_1(\mt F)-\theta
 \eta_1(\mt F)$, where $\bar\phi_1,\eta_1\colon \R^{3\times 3}\to\R$ are smooth and such that
\[
\begin{split}
\min_{\R^{3\times3}}\phi_1(\cdot,\theta) =   \min\Bigl\{0, -\alpha\Bigl(1-\frac{\theta}{\theta_T}\Bigr)\Bigr\},%
 \quad &
 \\
\bar\phi_1(\mt F) =-\alpha,\qquad\text{and}\qquad
\eta_1(\mt F)= - \frac{\alpha}{\theta_T},\quad &\text{for every } \mt F\in K,\\
\bar\phi_1(\mt F)=0\qquad \text{and}\qquad\eta_1(\mt F) 
=  0 ,\quad&\text{for every } \mt F\in SO(3),\\
\end{split}
\]
for every $\mt F\in \R^{3\times3}$, for some positive constants $\alpha$ {representing the latent heat of the transformation}. 
This term is responsible for the phase transition, that is, this term is responsible for changing the global minimizers of $\phi^k$ from $SO(3)$ above $\theta_c$ to $K$ below $\theta_c$. The behaviour is chosen to be linear in $\theta$ because of the small hysteresis;
\item the following growth conditions hold
\beq
\label{stime crescita}
\begin{split}
\bigl|\eta_1(\mt F)\bigr| + \bigl|\bar\phi_1(\mt F)\bigr|+\Bigl|\frac{\partial\phi_2}{\partial\mt F}(\mt F)\Bigr|+
\Bigl|\frac{\partial\bar\phi_1}{\partial\mt F}(\mt F)\Bigr|\leq c_d(1+|\mt F|^r),\\
\Bigl|\frac{\partial\eta_1}{\partial\mt F}(\mt F)\Bigr|\leq c_f(1+|\mt F|^{\tilde{r}}),
\end{split}
\eeq
for some positive constants 
$c_d,c_f$, $r <p$, $\tilde{r}<\frac{5}{6}p$ and for all $\mt F\in\R^{3\times3}$.
\end{itemize}
The choice of our free energy is coherent with the one adopted in \cite{AberKim} and references therein, 
and works particularly well for computations; other choices for the free energy 
reflecting the change of energetically preferable phase across $\theta_T$ would be possible. We finally add to the energy of the system also a term of the type $\frac1k|\Delta\vc y|^2$, which penalizes interfaces between martensitic variants. A term of this type is often kept in account when modelling martensitic transformations (see e.g., \cite{BallJames1,ContiSchweizer,FDP4,MullerDolzman}) and makes the model more accurate. This term is going to disappear from the equations when we send $k\to\infty$, in the next section. We remark that we could choose $\frac {c_\beta} {k ^\beta}$ for some $c_\beta,\beta>0$ in front of $|\Delta\vc y|^2$ instead of $\frac1k$ without affecting the results below. Furthermore we can replace $|\Delta\vc y|^2$ with $|\nabla\nabla\vc y|^2$. Indeed, as these terms differ only by a null Lagrangian, the equations governing the evolution are the same. 
\begin{remark}
{\rm
In the definition of the energy we did not take into account the constraint $\phi^k(\mt F,\cdot )\to\infty$ as $\det\mt F \to 0$, which is usually introduced to avoid interpenetration. This constraint is extremely difficult to handle from a mathematical point of view, as, in this case, there exists no $c>0$ such that $\Bigl|\frac{\partial\phi^k(\mt F,\cdot )}{\partial\mt F}\Bigl|\leq c(1+|\phi^k(\mt F,\cdot )|)$. That is, a finite energy deformation map $\vc y\in W^{1,p}(\Omega;\R^3)$ might be such that $\frac{\partial\phi^k(\nabla\vc y,\cdot )}{\partial\mt F}\notin L^1(\Omega;\R^{3\times 3})$, and there is, up to our knowledge, no reference in the literature on how to control this term. We refer the reader to Section 2.4 in \cite{BallOpen}. Nonetheless, the interpenetration constraint holds in the limit $k\to\infty$, as shown in Section \ref{inelali}.
}
\end{remark}
}

Under the above assumptions, and after adding the term responsible for the interface energy, \eqref{Pb 0} becomes
\beq
\label{Pb 1}
\begin{cases}
\rho_0\ddot{\vc y} - \Div \upsigma^k(\nabla \vc y,\theta)+\frac1k\Delta^2\vc y&=\vc 0,\\
\gamma\rho_0 \dot \theta -\Div(\mt C\nabla\theta) + \theta\dot\eta_1(\nabla \vc y) &= 0
,
\end{cases}
\eeq
where 
$$ \upsigma^k(\mt F,\theta) := \frac{\partial \phi^k}{\partial \mt F}(\mt F,\theta),\qquad  \eta_1(\mt F):= -\frac{\partial \phi_1} {\partial \theta}(\mt F).$$
We supplement the problem with initial and boundary conditions for $\theta$ and for $\vc y$ reflecting the experimental setup for Zn\textsubscript{45}Au\textsubscript{30}Cu\textsubscript{25} in \cite{JamesNew}. The initial conditions $\theta_0(\vc x),\vc y_a(\vc x),\vc y_b(\vc x)$ respectively represent from a physical point of view the temperature, the position and the velocity at the point $\vc x$ at $t=0$. The boundary conditions for the equation describing the conservation of momentum reflect the lack of external forces imposed at the boundary of the sample. {Different boundary conditions should be chosen to describe stress induced transformations.} We consider $\Omega\subset\R^3$ to be bounded, connected and smooth, $\partial\Omega_D\subset \partial\Omega$ to be open and non-empty in $\partial\Omega$, $\partial\Omega_N=\partial\Omega\setminus\partial\Omega_D$, and set
\beq
\label{Pb bc}
\begin{cases}
\theta=\theta_B,\qquad &\text{on } \partial\Omega_D\times(0,T),\\
\mt C\nabla\theta{\cdot \vc m}=0,\qquad &\text{on } \partial\Omega_N\times(0,T),\\
\bigl(\upsigma^k(\nabla \vc y,\theta)+\nabla\Delta\vc y	\bigr)\cdot \vc m= \vc 0,\qquad &\text{on } \partial\Omega\times(0,T),\\
\Delta\vc y	= \vc 0,\qquad &\text{on } \partial\Omega\times(0,T),\\
\theta|_{t=0}=\theta_0(\vc x),\qquad &\text{in } \Omega,\\
\vc y|_{t=0} = \vc y_a(\vc x),\qquad &\text{in } \Omega,\\
\dot{\vc y}|_{t=0} = \vc y_b(\vc x),\qquad &\text{in } \Omega,
\end{cases}
\eeq
where 
$\theta_B$, $\theta_0$, $\vc y_b$, $\vc y_a$ are functions in suitable function spaces, and $\vc m$ is the unit outer normal to $\Omega$. 
{Small thermal-hysteresis implies that the thermally induced phase transition can occur within a small temperature range around $\theta_T.$ For this reason, in a consistent way with the experiments in \cite{JamesNew} on Zn\textsubscript{45}Au\textsubscript{30}Cu\textsubscript{25} for which $\theta_T\approx240\,^\circ \mathrm K$ and the hysteresis is about $2\,^\circ \mathrm K$, we can assume $\|\theta_B-\theta_T\|_{L^\infty}+\|\theta_T-\theta_0\|_{L^\infty} \ll \theta_T $. 
%
%
%
%
%
} 
In this way, we can formally justify the replacement of the $\theta$ in front of $\dot\eta_1$ with $\theta_T$, thus rewriting \eqref{Pb 1} as
\beq
\label{cons ener}
\begin{cases}
\rho_0\ddot{\vc y} - \Div \upsigma^k(\nabla \vc y,\theta)+\frac1k\Delta^2\vc y&=\vc 0,
\\
\gamma\rho_0 \dot \theta -\Div(\mt C\nabla\theta) +\theta_T\dot\eta_1(\nabla \vc y)&= 0
.
\end{cases}
\eeq
Under suitable assumptions on $\eta_1$, a maximum principle such as the one in Proposition \ref{MacPrin} could be proved to rigorously justify this approximation.

\section{Existence of weak solutions}
\label{existence full}
{
Before proving existence of suitably defined solutions to \eqref{Pb bc}--\eqref{cons ener} we first introduce the notation for norms and functional spaces. 
\subsection{Notation}
Below $c,c_T$ represent two positive constants whose value is dependent just on the parameters, and, in the case of $c_T$, on the final time $T$ and the initial data. Their value may change from line to line or even within the same line, but is always independent of $n,k,\eps,M,N$. When there is a dependence on $k$ or on $M$, this constant will be denoted by $c_{T,k}$, $c_{T,M}$ respectively.

We denote by $\mt A :\mt B$ the Frobenius product between matrices $\mt A,\mt B\in\R^{3\times 3}$ defined as $\mt A:\mt B=\tr(\mt A^T\mt B)$. 

For $q\geq 1$ and $m\in\mathbb{N}$, $L^q(\Omega;\R^m)$ represents the space of Lebesgue measurable maps $\vc f\colon\Omega\to\R^m$ such that $\|\vc f\|_q:=\bigl(\mint |\vc f(\vc x)|^q\,\mathrm{d}\vc x\bigr)^{\frac1q}<\infty.$ As usual, if $q=\infty$, $\|\vc f\|_\infty = \esssup_\Omega \vc f.$ Whenever the space $L^s(\Omega;\R^m)$ is endowed with the weak topology, we denote it by $L^q_w(\Omega;\R^m)$. Given $s\in\mathbb{N}$ we also denote by
$W^{s,q}(\Omega;\R^m)$ (with the usual norm $\|\cdot\|_{W^{s,p}}$) the Sobolev space of maps $\vc f\in L^q(\Omega;\R^m)$ such that all derivatives of $\vc f$ up to order $s$ are in $L^q(\Omega)$. We follow the usual notation $H^s(\Omega;\R^m) = W^{s,2}(\Omega;\R^m)$ and we denote by $\|\cdot\|_{H^s}$ its norm. Below, we write $H^1_D$ the space of functions $\xi\in H^1(\Omega)$ such that $\xi=0$ on $\partial\Omega_D$, and by $H^{-1}_D$ its dual. Finally, we will denote by $\vc V$ the Hilbert space
$$
\vc V : = \bigl\{\vc u\in L^2(\Omega;\R^3)\colon \|\vc u\|_{\vc V} <\infty \bigr\},\qquad \|\vc u\|_{\vc V}^2:=\mint\bigl(|\vc u|^2+|\nabla\vc u|^2+|\Delta\vc u|^2	\bigr)\,\mathrm{d}\vc x,
$$ 
endowed with the scalar product $(\cdot,\cdot)_{\vc V}$ defined by
$$
(\vc u,\vc w)_{\vc V} := \mint\bigl(\vc u\cdot\vc w+\nabla\vc u:\nabla\vc w+\Delta\vc u\cdot\Delta\vc w\bigr)\,\mathrm{d}\vc x.
$$
Given the boundary conditions for our problem, we are not able to apply the Miranda-Talenti inequality, and therefore the embedding $\vc V\hookrightarrow H^1(\Omega;\R^3)$ might not be compact. The following lemma gives us sufficient conditions for this to hold:
\begin{lemma}
\label{lemmacomp0}
Let $\vc u_j\in \vc V$ be such that $\|\vc u_j\|_V \leq c$
and such that $|\nabla\vc u_j|^2$ is uniformly integrable. Then, there exists $\vc u\in\vc V$ and a non-relabelled subsequence $\vc u_j$ converging weakly in $\vc V$ and strongly in $H^1(\Omega;\R^3)$ to $\vc u$. Furthermore, the space $\vc V\cap W^{1,q}(\Omega;\R^3)$ is compactly embedded in $H^1(\Omega;\R^3)$ for every $q>2$.
\end{lemma}
\begin{proof}
By Banach-Alaoglu, we get the existence of $\vc u\in\vc V$ and of a non-relabelled subsequence $\vc u_j$ converging weakly in $\vc V$. By Sobolev embeddings, $\vc u_j\to\vc u$ also strongly in $L^2(\Omega;\R^3)$. Now, we notice that, on the one hand, $\nabla \times \nabla(\vc u-\vc u_j) = 0$ in the sense of distributions. On the other hand $\|\Div \nabla(\vc u-\vc u_j)\|_2\leq c$. Therefore, an application of a version of the div-curl lemma (see e.g., \cite{All}) entails 
$$
\lim_j\|\nabla\vc u_j-\nabla\vc u\|^2_2 =
\lim_j\int_{\Omega}  \nabla(\vc u-\vc u_j):\nabla(\vc u-\vc u_j)\,\mathrm{d}\vc x = 0.
$$
Now, we recall that for a sequence $\vc u_j\in W^{1,q}(\Omega;\R^3)$ with $\|\vc u_j\|_{W^{1,q}}\leq c$, for some $c>0$ $q>2$, $|\nabla\vc u_j|^2$ is uniformly integrable. Therefore, thanks to the first statement of the lemma, we deduce that $\vc V\cap W^{1,p}(\Omega;\R^3)$ is compactly embedded in $H^{1}(\Omega;\R^3).$
\end{proof}
Finally, given a Banach space $X$ with norm $\|\cdot\|_X$, $T>0$, and some $q\geq 1$, we use the space $L^s(0,T;X)$ of measurable maps $f\colon (0,T)\to X$ such that $\|f\|_{L^s(0,T;X)}:= \bigl( \Tint \|f(t)\|_X^s\,\mathrm{d}t\bigr)^{\frac1s}<\infty.$ We refer the reader to \cite{Boyer,Evans} for more details on these functional spaces.
}
\subsection{Statement and proof of the existence result}
We are now ready to prove the existence of suitably defined solutions to our problem:
\begin{theorem}
\label{esiste 3d}
Let $\theta_B\in H^{1}(\Omega) \cap L^{\frac{p}{p-r}}(\Omega)$, $\theta_0\in L^2(\Omega;\R_+)$, $\vc y_b\in L^2(\Omega;\R^3)$, $\vc y_a\in W^{1, p}(\Omega;\R^3)\cap \vc V$. Then, there exist $(\vc y_k,\theta_k)$ such that for every $T>0$ 
\begin{align*}
&\vc y_k \in L^\infty(0,T;W^{1,p}(\Omega;\R^3)\cap \vc V)\cap W^{1,\infty}(0,T; L^2(\Omega;\R^3)),\\
&\theta_k \in L^2(0,T;H^1(\Omega))\cap L^\infty(0,T;L^2(\Omega)),
\end{align*}
and
\begin{align}
\label{equazione 1a}
\Tint\mint \Bigl( \zeta\mt C\nabla\theta_k\cdot\nabla\omega - \bigl(\theta_T\eta_1(\nabla\vc y_k)+\gamma\rho_0\theta_k\bigr)\dot{\zeta}\omega \Bigr)\,\mathrm d\vc x\, \mathrm dt+	\zeta(0)\mint\omega\bigl(\gamma\rho_0\theta_0+\theta_T\eta_1(\nabla\vc y_a)\bigr)\,\mathrm d\vc x, 
\\
\label{equazione 1b}
\Tint\mint \biggl(\frac\xi k\Delta\vc y\cdot\Delta\boldsymbol\psi+\xi \nabla\boldsymbol\psi:\frac{\partial\phi^k}{\partial\mt F}(\nabla\vc y_k,\theta_k) -\rho_0\dot{\xi}\dot{\vc y}_k\cdot\boldsymbol \psi\biggr)\,\mathrm d\vc x\,\mathrm d t=
\rho_0\xi(0)\mint \vc y_1\cdot\boldsymbol\psi\,\mathrm{d}\vc x,
\end{align}
for all $\xi,\zeta\in C^1_c(-T,T)$, $\boldsymbol\psi\in C^1(\Omega;\R^3)\cap \vc V$, and all $\omega\in C^1(\Omega)\cap H^1_D$. Furthermore, $\theta_k=\theta_B$ on $\partial\Omega_D$ in the sense of traces for a.e. $t\in(0,T)$, 
$$\dot{\vc y}_k\in C([0,T];L^2_w(\Omega;\R^3)),\qquad\theta_k+\eta(\nabla\vc y_k)\in C([0,T]; L^{\min{\{2,p/{r}}\}}_w(\Omega))$$ 
and
\begin{align}
\label{che bund 3d}
\begin{split}
\Bigl(\|\theta_k\|^2_{2}+\|\dot{\vc y}_k\|_{2}^2+ \|\vc{y}_k\|^p_{W^{1,p}}+\frac1k\|\Delta\vc y_k\|_{2}^2\Bigr)(t)+\tint\|\nabla{\theta}_k\|^2_{2}(\tau)\,\mathrm{d}\tau 
 + k \mint \phi_2(\nabla \vc y_k(\vc x,t))\,\mathrm d\vc x \\\leq C_T+c\bigg|\mint \vc y_a (\vc x)\,\mathrm{d}\vc x\bigg|^p+\frac ck\|\Delta\vc y_a\|_{2}^2 + c k\mint \phi_2(\nabla \vc y_a) \,\mathrm d\vc x, 
\end{split}
\end{align}
for a.e $t\in(0,T)$, for some positive $C_T$ depending on the parameters and on $\|\theta_B\|_{{\frac{p}{p-r}}}$, $\|\theta_B\|_{H^1}$, $\|\theta_0\|_2$, $\|\vc y_b\|_2$ and $T$ only, and for some $c$ depending solely on the parameters.
\end{theorem}
\begin{proof}
Let $\{\omega_j\}_{j\in\mathbb{N}}\subset H^1_D(\Omega)$ and $\{\boldsymbol\psi_j\}_{j\in\mathbb{N}}\subset H^2(\Omega;\R^3)$ be bases of $H^1(\Omega)$ and $H^2(\Omega;\R^3)$ respectively. We assume that the first is constructed by taking the eigenvectors of the Laplacian with homogeneous Dirichlet boundary conditions on $\partial\Omega_D$ and homogeneous Neumann on $\partial\Omega_N$. The second one is constructed by taking the eigenvectors of the compact, symmetric, and linear operator $A\colon H^2(\Omega;\R^3)\to (H^2(\Omega;\R^3))^*$, where $$\langle A\vc u,\vc w\rangle = (\vc u,\vc w)_{H^2},
$$ for each $\vc u,\vc w \in H^2(\Omega;\R^3).$ Here $(\cdot,\cdot)_{H^2}$ is the scalar product in the space $H^2(\Omega;\R^3).$ 
We can define $P_n$ and $\vc P_n$ to be respectively the $L^2(\Omega)$ and the $L^2(\Omega;\R^3)$ projectors on the finite subspaces $\{\omega_j\}_{j=1,\dots,n}$ and $\{\boldsymbol\psi_j\}_{j=1,\dots,n}$. Let us consider $k$ to be fixed, we drop the subscript $k$ to simplify notation, and we suppose without loss of generality that $\rho_0=\gamma=\theta_T=1$. Furthermore, we start by assuming that $\vc y_a\in H^2(\Omega;\R^3)\cap W^{1,p}(\Omega;\R^3)$. The strategy of the proof is the following: we show existence of solutions via the Galerkin method to an approximated problem depending on parameters $M,N\in\mathbb{N}.$ Then we send first $M$ and then $N$ to infinity, and show that solutions to the approximated problems converge to solutions to the original problem. Finally, we recover the result for general $\vc y_a\in \vc V\cap W^{1,p}(\Omega;\R^3)$.\\

\noindent
For every $n\geq 1$ let us consider the following functions
\[
\theta_{n}(\vc x,t)=\theta_B+\sum_{j=1}^{n} b_j^{(n)}(t)\omega_j(\vc x),\qquad
\vc y_{n}(\vc x,t)=\sum_{j=1}^{n} c_j^{(n)}(t)\boldsymbol\psi_j(\vc x),
\]
satisfying
\begingroup
\addtolength{\jot}{0.9em}
\begin{align}
\label{approx 1a}
&\mint \bigl(\dot\theta_n\omega_j+ \mt C\nabla\theta_n\cdot\nabla\omega_j \bigr)\,\mathrm{d}\vc x = - \mint \dot\eta_1(H_M(\nabla\vc y_n))\omega_j\,\mathrm d\vc x, \qquad \text{for all }j=1,\dots,n,
\\
\label{approx 1b}
\begin{split}
&\mint \biggl(\frac{\partial\phi^k}{\partial\mt F}(H_M(\nabla\vc y_n),\theta_n)\frac{\partial H_M(\nabla\vc y_n)}{\partial\mt F}:\nabla\boldsymbol\psi_i+\frac{1}{N}\bigl(\vc y_n,\boldsymbol\psi_i\bigr)_{H^2}\biggr)\,\mathrm d\vc x 
\\
&=- \mint\Bigl(\frac1k\Delta\boldsymbol\psi_i\cdot\Delta\vc y_n
+\ddot{\vc y}_n\cdot\boldsymbol \psi_i\Bigr)\,\mathrm d\vc x, \qquad \text{for all }i=1,\dots,n,
\end{split}
\\\bigskip
\nonumber
&\theta_n(\vc x, 0) = P_n(\theta_0(\vc x)-\theta_B)+\theta_B,\qquad \vc y_n(\vc x, 0) = \vc P_n\vc y_a(\vc x),\qquad \dot{\vc y}_n(\vc x, 0) = \vc P_n\vc y_b(\vc x),
\end{align}
\endgroup
for almost every $\vc x\in\Omega$, for $M,N\in\mathbb{N}$ fixed, and where $H_M$ is a smooth function such that 
\beq
\label{defH}
H_M(\mt F) = \begin{cases}
\mt F,\qquad&\text{if $|\mt F|\leq M$},\\
(M+1) \frac{\mt F}{|\mt F|},\qquad&\text{if $|\mt F|\geq M+1$,}
\end{cases}
\eeq
and $\Bigl|\frac{\partial H_M(\mt F)}{\partial\mt F}\Bigr|\leq 2$. We remark that $\frac{\partial H_M(\mt F)}{\partial\mt F}$ is a fourth order tensor, and $\frac{\partial \phi^k(H_M(\mt G),\cdot)}{\partial\mt F}\frac{\partial H_M(\mt G)}{\partial\mt F}$ should read as
$$
\Bigl(\frac{\partial \phi^k(H_M(\mt G),\cdot)}{\partial\mt F}\frac{\partial H_M(\mt G)}{\partial\mt F}\Bigr)_{lm} =\sum_{i,j=1}^3 \frac{\partial \phi^k(\mt H,\cdot)}{\partial\mt H_{ij}}\Bigl|_{\mt H=H_M(\mt G)}\frac{\partial (H_M(\mt F))_{ij}}{\partial\mt F_{lm}} \Bigl|_{\mt F=\mt G}= \frac{\partial \phi^k(H_M(\mt F),\cdot)}{\partial\mt F_{lm}}\Bigl|_{\mt F=\mt G}.
$$
The above system of differential equations in $\vc b^{(n)},\vc c^{(n)}$ admits locally in time a smooth solution by standard ODE theory. Define $\hat{\theta}_n:=\theta_n-\theta_B$ and $\hat{\theta}_{0,n}:= P_n(\theta_0(\vc x)-\theta_B)$. After multiplying \eqref{approx 1a} tested against $\omega_j$ by $b_j^{(n)}$, we can sum over $j$ and get
\beq
\label{es3d 1}
\frac12\dert\|\hat{\theta}_n\|^2_2+c_m\|\nabla\hat{\theta}_n\|^2_2\leq  -\mint\dot\eta_1(H_M(\nabla\vc y_n))\hat{\theta}_n\,\mathrm d\vc x - \mint\nabla\theta_B\cdot\nabla\hat{\theta}_n.
\eeq
Here we have also used the fact that $c_m$ is the minimum eigenvalue of $\mt C$ to bound $\mt C\nabla\hat{\theta}_n\cdot\nabla\hat{\theta}_n$ from below with $c_m|\nabla\hat{\theta}_n|^2$. Thanks to the Cauchy-Schwarz and the Young inequalities, \eqref{es3d 1} becomes
$$
\frac12\dert\|\hat{\theta}_n\|^2_2+\frac{c_m}2\|\nabla\hat{\theta}_n\|^2_2\leq  -\mint\hat{\theta}_n\dot\eta_1(H_M(\nabla\vc y_n))\,\mathrm d\vc x + c\|\nabla\theta_B\|_2^2.
$$
On the other hand, testing \eqref{approx 1b} with $\boldsymbol \psi_i$, multiplying it by $\dot c_i^{(n)}$ and summing over $i$ from $1$ to $n$ we get
\[
\begin{split}
&\frac12\dert \biggl( \|\dot{\vc y}_n\|_2^2+\frac1k\|\Delta\vc y_n\|_2^2+\frac1N \|\vc y_n\|_{H^2}^2\biggr)+ \dert\mint \Bigl(k\phi_2(H_M(\nabla\vc y_n))+\bar{\phi}_1(H_M(\nabla\vc y_n))\Bigr)\,\mathrm d\vc x\\
&= \mint \theta_n\dot{\eta}_1(H_M(\nabla\vc y_n))\,\mathrm d\vc x.
\end{split}
\]
Therefore, summing the last two inequalities and integrating in time between $0$ and $t\in(0,T)$ we get
\begingroup
\addtolength{\jot}{0.8em}
\beq
\label{es3d 2}
\begin{split}
\frac12\biggl(\|\dot{\vc y}_n\|_2^2&+\|\hat{\theta}_n\|^2_2+\frac1k\|\Delta\vc y_n\|_2^2+\frac1N  \|\vc y_n\|_{H^2}^2\biggr)(t)+\frac{c_m}2\tint\|\nabla\hat{\theta}_n\|^2_2\,\mathrm d\tau
\\
&+ \mint \Bigl(k\phi_2(H_M(\nabla\vc y_n(t)))+\bar{\phi}_1(H_M(\nabla\vc y_n(t)))\Bigr)\,\mathrm d\vc x
 \\
\leq  
 \mint \Bigl(&k\phi_2(H_M(\nabla\vc P_n\vc y_a))+\bar{\phi}_1(H_M(\nabla\vc P_n\vc y_a))\Bigr)\,\mathrm d\vc x\\
 +\mint&\eta_1(H_M(\nabla\vc y_n(t)))\theta_B\,\mathrm d\vc x-\mint\eta_1(H_M(\nabla\vc P_n\vc y_a))\theta_B\,\mathrm d\vc x+ ct\|\nabla\theta_B\|_2^2\\ 
&+\frac12\biggl(\|\vc P_n{\vc y}_b\|_2^2+\|\hat{\theta}_{0,n}\|^2_2+\frac1k\|\Delta\vc P_n\vc y_{a}\|_2^2+\frac1N  \|\vc P_n\vc y_a\|_{H^2}^2\biggr)
. 
\end{split}
\eeq
\endgroup
Now, we notice that by means of the assumptions on $\eta$ we can deduce
\[
\begin{split}
\biggl|\mint\eta_1(H_M(\nabla\vc P_n\vc y_a))\theta_B\,\mathrm d\vc x\biggr|\leq 
c\|\theta_B\|_{\frac{p}{p-r}}\bigl(1+\|H_M(\nabla\vc P_n\vc y_a)\|^r_{p}
\bigr)
\leq c\bigl(1+\|\theta_B\|_{\frac{p}{p-r}}^{\frac{p}{p-{r}}}\bigr) + \frac{c_0}{4}\|H_M(\nabla\vc P_n\vc y_a)\|_p^p\\
\leq c\bigl(1+\|\theta_B\|_{\frac{p}{p-r}}^{\frac{p}{p-{r}}}\bigr) + \frac{k}{4}\phi_2(H_M(\nabla\vc P_n\vc y_a)).
\end{split}
\]
In the same way,
\begin{align*}
\biggl|\mint\eta_1(H_M(\nabla\vc y_n (t)))\theta_B\,\mathrm d\vc x\biggr|&\leq c\bigl(1+\|\theta_B\|_{\frac{p}{p-r}}^{\frac{p}{p-{r}}}\bigr) + \frac{k}{4}\mint\phi_2(H_M(\nabla\vc y_n(t)))\,\mathrm{d}\vc x,\\
\biggl|\mint\bar{\phi}_1(H_M(\nabla\vc P_n\vc y_a))\,\mathrm d\vc x\biggr|&\leq c + \frac{k}{4}\mint\phi_2(H_M(\nabla\vc P_n\vc y_a))\,\mathrm{d}\vc x,\\
\biggl|\mint\bar{\phi}_1(H_M(\nabla\vc y_n (t)))\,\mathrm d\vc x\biggr|&\leq c + \frac{k}{4}\mint\phi_2(H_M(\nabla\vc y_n(t)))\,\mathrm{d}\vc x.
\end{align*}
Thus, \eqref{es3d 2} becomes
\begingroup
\addtolength{\jot}{0.3em}
\beq
\label{super stima energia}
\begin{split}
\biggl(\frac1k\|\Delta\vc y_n\|_2^2+\|\dot{\vc y}_n\|_2^2&+\|\hat{\theta}_n\|^2_2+\frac1N  \|\vc y_n\|_{H^2}^2\biggr)(t)+{c_m}\tint\|\nabla\hat{\theta}_n\|^2_2\,\mathrm d\tau + k\mint \phi_2(H_M(\nabla\vc y_n(t)))\,\mathrm d\vc x
\\
\leq \frac1k\|\Delta&\vc P_n\vc y_a\|_2^2+\|\vc P_n{\vc y}_b\|_2^2+\|\hat{\theta}_{0,n}\|^2_2 
+ ct\|\nabla\theta_B\|_2^2+c\bigl(1+\|\theta_B\|_{\frac{p}{p-r}}^{\frac{p}{p-{r}}}\bigr)\\
&+ 3k\mint \phi_2(H_M(\nabla\vc P_n\vc y_a)) \,\mathrm d\vc x+\frac1N  \|\vc P_n\vc y_a\|_{H^2}^2
 . 
\end{split}
\eeq
\endgroup
We remark that as $\vc P_n\vc y_a\to\vc y_a$ strongly in $H^2(\Omega;\R^3)$, and $\vc P_n\vc y_b\to\vc y_b$ strongly in $L^2(\Omega;\R^3)$, the sequences $\|\vc P_n\vc y_a\|_{H^2}$, $\|\vc P_n\vc y_b\|_{2}$ are bounded. Therefore, by the boundedness of $H_M$(cf. \eqref{defH}), and thanks to the fact that $\phi_2$ is non-negative, we deduce that for every $T>0$ there exists a constant $\tilde C= \tilde C({T,M,k},N)$, independent of $n$, such that
\[
\frac1k\|\Delta\vc y_n\|_2^2(t)+\|\dot{\vc y}_n\|_2^2(t)+\|\hat{\theta}_n\|^2_2(t)+{c_m}\tint\|\nabla\hat{\theta}_n\|^2_2\,\mathrm d\tau +\frac{1}{N}  \|\vc y_n\|_{H^2}^2(t)
\leq \tilde C, 
\]
for a.e. $t\in(0,T)$. 
Hence, there exists 
\begin{align*}
&y_M\in L^\infty(0,T;H^2(\Omega;\R^3))\cap W^{1,\infty}(0,T;L^2(\Omega;\R^3)),\qquad
&\theta_M\in L^2(0,T;H^1(\Omega))\cap L^\infty(0,T;L^2(\Omega)),
\end{align*}  such that, up to a subsequence,
\begin{align}
\label{3d conv1}
&\te_n\tow \te_M \quad \text{weakly  in } L^2(0,T;H^1(\Omega)),\quad \text{weakly$*$  in } L^\infty(0,T;L^2(\Omega)),\\
\label{3d conv4}
&\vc y_n\tow \vc y_M \quad \text{weakly$*$  in } L^\infty(0,T;H^2(\Omega;\R^3))\cap W^{1,\infty}(0,T;L^2(\Omega;\R^3)).
\end{align}
Furthermore, \eqref{3d conv4} together with a version of the Aubin-Lions lemma (see e.g., \cite[Thm. II.5.16]{Boyer}) imply
\beq
\label{3d conv5}
\vc y_n\to\vc y_M,\quad\text{strongly in $C([0,T];H^1(\Omega;\R^3))$}
,
\eeq
that is, up to a further subsequence, $\nabla\vc y_n\to\nabla\vc y_M$ for each $t\in[0,T]$, a.e. $\vc x\in\Omega$. Thus, as $H_M$ is continuous and bounded, by Vitali's convergence theorem $H_M(\nabla \vc y_n)\to H_M(\nabla \vc y_M)$ strongly in $L^{\tilde q}(0,T;L^{\tilde q}(\Omega))$ for each $\tilde q\in [1,\infty).$ 

We can thus multiply \eqref{approx 1a} and \eqref{approx 1b} by $\zeta,\xi\in C^1_c(-T,T)$ and integrate in time between $0$ and $T$. 
After an integration by parts these become 
\begin{align}
\label{approx 1a bis}
\begin{split}
\Tint\mint &\bigl(\zeta \mt C\nabla\theta_n\cdot\nabla\omega_j - \omega_j\dot{\zeta}(\eta_1(H_M(\nabla\vc y_n))+\theta_n)\bigr)\,\mathrm{d}\vc x\,\mathrm dt \\
&= \zeta(0) \mint \bigl(\eta_1(H_M(\nabla\vc P_n\vc y_a))+\theta_{0,n}\bigr)\omega_j\,\mathrm d\vc x,
\end{split}
\\
\label{approx 1b bis}
\begin{split}
\Tint\mint &\biggl(\frac\xi k\Delta\vc y_n\cdot\Delta\boldsymbol\psi_i+\xi\frac{\partial\phi^k}{\partial\mt F}(H_M(\nabla\vc y_n),\theta_n)\frac{\partial H_M}{\partial\mt F}(\nabla\vc y_n):\nabla\boldsymbol\psi_i  \biggr)\,\mathrm d\vc x\,\mathrm dt 
\\ &=\Tint\mint\biggl(\dot\xi\dot{\vc y}_n\cdot\boldsymbol \psi_i-\frac\xi N \bigl(\boldsymbol\psi_i,\vc y_n\bigr)_{H^2} \biggr)\,\mathrm d\vc x\,\mathrm dt + \xi(0) \mint \vc P_n\vc y_b\cdot\boldsymbol\psi_i\,\mathrm d\vc x,
\end{split}
\end{align}
for each $i,j=1,\dots,n$. Passage to the limit and recovering the equations for every $\omega\in H^1_D(\Omega)\cap C^1(\Omega)$, $\boldsymbol\psi\in C^1(\Omega;\R^3)\cap H^2(\Omega;\R^3)$ is standard. Thus, \eqref{approx 1a bis}--\eqref{approx 1b bis} become
\begin{align}
\label{approx 1a tris}
\Tint\mint &\bigl(\zeta \mt C\nabla\theta_M\cdot\nabla\omega - \omega\dot{\zeta}(\eta_1(H_M(\nabla\vc y_M))+\theta_M)\bigr)\,\mathrm{d}\vc x\,\mathrm dt = \zeta(0) \mint \bigl(\eta_1(H_M(\nabla \vc y_a))+\theta_{0}\bigr)\omega\,\mathrm d\vc x,
\\
\label{approx 1b tris}
\begin{split}
\Tint\mint &\biggl(\frac\xi k\Delta\vc y_M\cdot\Delta\boldsymbol\psi+\xi\frac{\partial\phi^k}{\partial\mt F}(H_M(\nabla\vc y_M),\theta_M)\frac{\partial H_M}{\partial\mt F}(\nabla\vc y_M):\nabla\boldsymbol\psi  \biggr)\,\mathrm d\vc x\,\mathrm dt 
\\ &=\Tint\mint\biggl(\dot\xi\dot{\vc y}_M\cdot\boldsymbol \psi-\frac\xi N \bigl(\boldsymbol\psi_i,\vc y_n\bigr)_{H^2}\biggr)\,\mathrm d\vc x\,\mathrm dt + \xi(0) \mint \vc y_b\cdot\boldsymbol\psi\,\mathrm d\vc x,
\end{split}
\end{align}
for every $\omega\in H^1_D(\Omega)\cap C^1(\Omega)$, $\boldsymbol\psi\in C^1(\Omega;\R^3)\cap \vc H^2(\Omega;\R^3)$, $\zeta,\xi\in C^1_c(-T,T)$. {Also, from \eqref{super stima energia} we can write
\[
\begin{split}
\biggl(\frac1k\|\Delta\vc y_n\|_2^2+\|\dot{\vc y}_n\|_2^2+\|\hat{\theta}_n\|^2_2+\frac1N  \|\vc y_n\|_{H^2}^2\biggr)(t)+{c_m}\tint\|\nabla\hat{\theta}_n\|^2_2\,\mathrm d\tau + k\mint \phi_2(H_M(\nabla\vc y_n(t)))\,\mathrm d\vc x
\\
\leq
\frac1k\|\Delta\vc y_n\|_{L^\infty(0,T;L^2)}^2+
\frac{1}{N}\| \vc y_n\|_{L^\infty(0,T;H^2)}^2+\|\dot{\vc y}_n\|^2_{L^\infty(0,T;L^2)}
+\|\hat{\theta}_n\|^2_{L^\infty(0,T;L^2)}+c_m\|\nabla\hat{\theta}_n\|^2_{L^2(0,T;L^2)}
\\+\esssup_{t\in(0,T)}k\mint \phi_2(H_M(\nabla\vc y_n(t)))\,\mathrm d\vc x
\leq 6\Bigl(\frac1k\|\Delta\vc P_n\vc y_a\|_2^2
+ c_T
+ 3k\mint \phi_2(H_M(\nabla\vc P_n\vc y_a)) \,\mathrm d\vc x+\frac1N  \|\vc P_n\vc y_a\|_{H^2}^2\Bigr).
\end{split}
\]
Therefore, a passage to the limit as $n\to\infty$ in \eqref{super stima energia} and the weak$*$ lower semicontinuity of the norms leads to
\beq
\label{energia inmezzoN}
\begin{split}
\frac1k\|\Delta\vc y_M\|_{L^\infty(0,T;L^2)}^2+
\frac{1}{N}\| \vc y_M\|_{L^\infty(0,T;H^2)}^2+\|\dot{\vc y}_M\|^2_{L^\infty(0,T;L^2)}
+\|\hat{\theta}_M\|^2_{L^\infty(0,T;L^2)}+c_m\|\nabla\hat{\theta}_M\|^2_{L^2(0,T;L^2)}
\\+\esssup_{t\in(0,T)}\mint \phi_2(H_M(\nabla\vc y_M(t)))\,\mathrm d\vc x
\leq 6\Bigl(c_T+\frac1N\|\vc y_a\|^2_{H^2} +\frac1k\|\Delta\vc y_a\|^2_2
+ 3k \mint \phi_2(H_M(\nabla \vc y_a))\,\mathrm d\vc x\Bigr)
.
 \end{split} 
\eeq
Here we also used the fact that, defined $f_n:=\mint \phi_2(H_M(\nabla\vc y_n(t)))\,\mathrm d\vc x$, we have that, up to a further subsequence $f_n$ converges weakly$*$ to some $f\in L^\infty(0,T)$, and therefore $\esssup_{t\in(0,T)} f_n \geq \esssup_{t\in(0,T)} f.$ But since for each $t\in[0,T]$ we have $\nabla\vc y_n\to\nabla \vc y_M$ a.e., we also have 
$$
\mint \phi_2(H_M(\nabla\vc y_n(t)))\,\mathrm d\vc x\to \mint \phi_2(H_M(\nabla\vc y_M(t)))\,\mathrm d\vc x,\qquad\text{for each $t\in[0,T]$,}
$$
and thus $f = \mint \phi_2(H_M(\nabla\vc y_M(t)))\,\mathrm d\vc x$ as claimed.
} 
Now, given the fact that $\vc y_a\in W^{1,p}(\Omega;\R^3)$, we have that $H_M(\nabla\vc y_a)$ converges strongly to $\nabla\vc y_a$ as $M\to\infty$ by dominated convergence in $L^p$, and hence, by \eqref{stime crescita}, $\phi_2(H_M(\nabla\vc y_a))$ converges strongly in $L^1$ to $\phi_2(\nabla\vc y_a)$. Therefore, inequality \eqref{energia inmezzoN} yields 
\beq
\label{big bund2}	
\begin{split}
\| \vc y_M\|_{L^\infty(0,T;H^2)}+\|\dot{\vc y}_M\|_{L^\infty(0,T;L^2)}
+\|\hat{\theta}_M\|_{L^\infty(0,T;L^2)}+\|\nabla\hat{\theta}_M\|_{L^2(0,T;L^2)}
\\+\esssup_{t\in(0,T)}\mint \phi_2(H_M(\nabla\vc y_M(t)))\,\mathrm d\vc x
\leq c_{T,k,N}. 
\end{split}
\eeq
We can hence deduce the existence of
$$
\vc y_N\in L^\infty(0,T;H^2(\Omega;\R^3))\cap W^{1,\infty}(0,T;L^2(\Omega;\R^3)),\qquad
\theta_N\in L^2(0,T;H^1(\Omega))\cap L^\infty(0,T;L^2(\Omega)),
$$
such that, up to a subsequence, $\theta_M\to\theta_N$ and $\vc y_M\to\vc y_N$ as $M\to\infty$ in the sense of \eqref{3d conv1}--\eqref{3d conv5}. Furthermore, by \eqref{big bund2} and \eqref{stime crescita} we also have
\beq
\label{bb 2}
\|H_M(\vc y_M)\|_{L^\infty(0,T;L^p)}
\leq c_{T,k,N}. 
\eeq
Therefore, as $H_M(\nabla\vc y_M)\to \nabla\vc y_N$ for a.e. $(x,t)\in\Omega\times(0,T)$, by Vitali's theorem we deduce
\beq 
\label{unaltraconv}
H_M(\nabla \vc y_M)\to \nabla\vc y_N,\quad\text{strongly in $L^{\bar q}(\Omega\times(0,T))$}
,
\eeq
for every $\bar q\in(1,p)$. We remark that Fatou's lemma and the fact that $H_M(\nabla \vc y_M)\to \nabla\vc y_N$ for a.e. $(\vc x,t)\in\Omega\times(0,T)$, imply
$$
\mint \phi_2(\nabla\vc y_N(\vc x,t)) \,\mathrm d\vc x\leq \liminf_{M\to\infty}\mint \phi_2(H_M(\nabla \vc y_M(\vc x,t))) \,\mathrm d\vc x,\qquad\text{for a.e $t\in(0,T)$}.
$$ 
This inequality, together with the lower semicontinuity of the norms, and the convergences \eqref{3d conv1}--\eqref{3d conv5} entail
\beq
\label{energia inmezzo N2}
\begin{split}
\frac1k\|\Delta\vc y_N\|_{L^\infty(0,T;L^2)}^2+
\frac{1}{N}\| \vc y_N\|_{L^\infty(0,T;H^2)}^2+\|\dot{\vc y}_N\|^2_{L^\infty(0,T;L^2)}
+\|{\theta}_N\|^2_{L^\infty(0,T;L^2)}+\|\nabla{\theta}_N\|^2_{L^2(0,T;L^2)}
\\+\esssup_{t\in(0,T)} k \mint \phi_2(\nabla\vc y_N(t))\,\mathrm d\vc x
\leq c\Bigl(C^*+\frac1N\|\vc y_a\|^2_{H^2} +\frac1k\|\Delta\vc y_a\|^2_2
+ 3k \mint \phi_2(\nabla \vc y_a)\,\mathrm d\vc x\Bigr),
 \end{split} 
\eeq
for some $C^*>0$ depending on $\|\theta_B\|_{H^1},\|\theta_B\|_{\frac{p}{p-r}},\|\vc y_b\|_{2}$, $\|\theta_0\|_2$ and $T$ only. By \eqref{stime crescita}, together with \eqref{energia inmezzo N2}, we thus have also 
$$\vc y_N \in L^\infty(0,T;W^{1,p}(\Omega;\R^3)).$$
The assumptions in \eqref{stime crescita}, and the bound on $\frac{\partial H_M}{\partial\mt F}$, together with \eqref{unaltraconv} imply also that
\begin{align}
\label{convergenze derivate 1}
\frac{\partial\phi_2}{\partial\mt F}(H_M(\nabla \vc y_M))\frac{\partial H_M}{\partial\mt F}(\nabla \vc y_M)\to 
\frac{\partial\phi_2}{\partial\mt F}(\nabla \vc y_N),&\qquad\text{strongly in $L^1(\Omega;\R^{3\times3})$,}\\
\label{convergenze derivate 2}
\frac{\partial\bar\phi_1}{\partial\mt F}(H_M(\nabla \vc y_M))\frac{\partial H_M}{\partial\mt F}(\nabla \vc y_M)\to 
\frac{\partial\bar\phi_1}{\partial\mt F}(\nabla \vc y_N),&\qquad\text{strongly in $L^1(\Omega;\R^{3\times3})$.}
\end{align}
Hence, passage to the limit as $M\to\infty$ in \eqref{approx 1a tris}--\eqref{approx 1b tris} is standard thanks to \eqref{3d conv1}--\eqref{3d conv5} together with \eqref{convergenze derivate 1}--\eqref{convergenze derivate 2}. The only difficulty is to show that, defining $$R_M:=\frac{\partial\eta_1}{\partial\mt F}(H_M(\nabla\vc y_M))\frac{\partial H_M}{\partial\mt F}(\nabla\vc y_M),$$ we have
\[
\begin{split}
\biggl|\Tint\mint \xi\nabla\boldsymbol\psi&:\Bigl(\frac{\partial \eta _1}{\partial\mt F}
(\nabla\vc y_N)\theta_N - R_M\theta_M\Bigr)\,\mathrm{d}\vc x\,\mathrm dt\biggr|
\\
&=
\biggl|\Tint\mint \xi\nabla\boldsymbol\psi:\Bigl(\frac{\partial\eta_1}{\partial\mt F}(\nabla\vc y_N)(\theta_N-\theta_M)+ \theta_M  \Bigl(\frac{\partial\eta_1}{\partial\mt F}(\nabla\vc y_N) - R_M \Bigr) \Bigr)\,\mathrm{d}\vc x\,\mathrm dt\biggr|
\to 0,
\end{split}
\]
as $M\to\infty$, for every $\boldsymbol\psi\in C^1(\Omega;\R^3)\cap H^1(\Omega;\R^3)$, every $\xi\in C^1_c(-T,T)$. On the one hand,
$$
\biggl|\Tint\mint \xi\nabla\boldsymbol\psi:\Bigl(\frac{\partial\eta_1}{\partial\mt F}(\nabla\vc y_N)(\theta_N-\theta_M) \Bigr)\,\mathrm{d}\vc x\,\mathrm dt\biggr|
\to 0,\qquad\text{as $M\to\infty$,}
$$
because of the fact that $\theta_M\to\theta_N$ weakly in $L^2(0,T;L^6(\Omega))$ as $M\to\infty$, and that \eqref{stime crescita} together with $\nabla\vc y_N \in L^\infty(0,T;L^p(\Omega;\R^{3\times3}))$ imply $\frac{\partial\eta_1}{\partial\mt F}(\nabla\vc y_N) \in L^\infty(0,T;L^\frac65(\Omega))$. On the other hand, 
\[
\begin{split}
\biggl|\Tint\mint \xi\nabla\boldsymbol\psi &:\Bigl(\theta_M  \Bigl(R_M - \frac{\partial\eta_1}{\partial\mt F}(\nabla\vc y_N)\Bigr) \Bigr)\,\mathrm{d}\vc x\,\mathrm dt\biggr|\\
&\leq \|\theta_M\|_{L^2(0,T;L^6)}\|\boldsymbol\psi\|_{C^1}\|\xi\|_{C^1}\Bigl\|\frac{\partial\eta_1}{\partial\mt F}(\nabla\vc y_N) - R_M\Bigr\|_{L^2(0,T;L^{\frac65})}
\to 0
\end{split}
\]
because of \eqref{unaltraconv}. In the limit, \eqref{approx 1a tris}--\eqref{approx 1b tris} thus become
\begin{align}
\label{approx 1a quater}
\Tint\mint &\bigl(\zeta \mt C\nabla\theta_N\cdot\nabla\omega - \omega\dot{\zeta}(\eta_1(\nabla\vc y_N)+\theta_N)\bigr)\,\mathrm{d}\vc x\,\mathrm dt = \zeta(0) \mint \bigl(\eta_1(\nabla \vc y_a)+\theta_{0}\bigr)\omega\,\mathrm d\vc x,
\\
\label{approx 1b quater}
\begin{split}
\Tint\mint &\biggl(\frac\xi k\Delta\vc y_N\cdot\Delta\boldsymbol\psi+\xi\nabla\boldsymbol\psi :\frac{\partial\phi^k}{\partial\mt F}(\nabla\vc y_N,\theta_N)  \biggr)\,\mathrm d\vc x\,\mathrm dt 
\\ &=\Tint\mint\biggl(\dot\xi\dot{\vc y}_N\cdot\boldsymbol \psi-\frac\xi N \bigl(\boldsymbol\psi,\vc y_N\bigr)_{H^2} \biggr)\,\mathrm d\vc x\,\mathrm dt + \xi(0) \mint \vc y_b\cdot\boldsymbol\psi\,\mathrm d\vc x,
\end{split}
\end{align}
for every $\omega\in H^1_D(\Omega)\cap C^1(\Omega)$, $\boldsymbol\psi\in C^1(\Omega;\R^3)\cap H^2(\Omega;\R^3)$, $\zeta,\xi\in C^1_c(-T,T)$. 
Now, we want to send $N$ to infinity. To this end, we first suppose $N\geq N_*$, where $N_* = N_*(\|\vc y_a\|_{H^2})$ is the smallest integer such that $N_*^{-1}\|\vc y_a\|_{H^2}^2\leq 1$. In this case, by \eqref{energia inmezzo N2} we have
\beq
\label{altrodasootyn}
\|\nabla\vc y_N\|_p^p(t)-c\leq \mint \phi_2(\nabla\vc y_N)\,\mathrm{d}\vc x\leq c_{T,k},\qquad\text{a.e. $t\in(0,T)$}.
\eeq
Therefore, \eqref{altrodasootyn} together with \eqref{energia inmezzo N2} entail
\beq
\label{energia inmezzo N3}
\begin{split}
\Bigl(\frac1N\|\vc y_N\|_{H^2}^2+\|\Delta\vc y_N\|_2^2+\|\dot{\vc y}_N\|_2^2+\|\nabla\vc y_N\|_p^p+\|{\theta}_N\|^2_2\Bigr)(t)+{c_m}\tint\|\nabla{\theta}_N\|^2_2\,\mathrm d\tau
\leq c_{T,k}
,
 \end{split} 
\eeq
for a.e. $t\in(0,T)$. Let $\{\vc e_1,\vc e_2,\vc e_3\}$ be a Cartesian coordinate system for $\R^3$. We can choose $\boldsymbol\psi =\vc e_i$ with $i=1,\dots,3$ in \eqref{approx 1b quater} and deduce that
$$
\frac{d^2}{dt^2}\mint \vc y_N(t)\cdot\vc e_i\,\mathrm d\vc x = - \frac{1}{N}\mint \vc y(t)\cdot\vc e_i\,\mathrm d\vc x,
$$
for a.e. $t\in(0,T)$. This implies,
$$
\mint \vc y_N(t)\cdot\vc e_i\,\mathrm d\vc x = \cos\Bigl(\frac{t}{\sqrt N}\Bigr)\mint \vc y_a\cdot\vc e_i\,\mathrm d\vc x +\sqrt{N}\sin\Bigl(\frac{t}{\sqrt N}\Bigr)\mint \vc y_b\cdot\vc e_i\,\mathrm d\vc x,
$$
and therefore, for every $T$ finite, we deduce that
\beq
\label{bound media S}
\bigg|\mint \vc y_N(t)\cdot\vc e_i\,\mathrm d\vc x\bigg| \leq  \bigg|\mint \vc y_a\cdot\vc e_i\,\mathrm d\vc x\bigg| + T\bigg|\mint \vc y_b\cdot\vc e_i\,\mathrm d\vc x\bigg| \leq c_T,
\eeq
where, again, $c_T$ is independent of $N,k$. Therefore, an application of the Poincar\'e inequality, together with \eqref{altrodasootyn} leads also to
\beq
\label{N3 con P}
\|\vc y_N\|_{W^{1,p}}(t)\leq c_{T,k},\qquad\text{a.e. $t\in(0,T)$}.
\eeq
Therefore, from \eqref{energia inmezzo N3}--\eqref{N3 con P} we deduce the existence of 
$$
\vc y \in L^\infty(0,T;\vc V\cap W^{1,p}(\Omega;\R^3))\cap W^{1,\infty}(0,T;L^2(\Omega;\R^3)),\qquad
\theta \in L^2(0,T;H^1(\Omega))\cap L^\infty(0,T;L^2(\Omega)),
$$
and a non-relabelled subsequence, such that  
\begin{align}
\label{3d conv1 end}
&\te_N\tow \te \quad \text{weakly  in } L^2(0,T;H^1(\Omega)),\quad \text{weakly$*$  in } L^\infty(0,T;L^2(\Omega)),\\
\label{3d conv4 end}
&\vc y_N\tow \vc y \quad \text{weakly$*$  in } L^\infty(0,T;\vc V\cap W^{1,p}(\Omega;\R^3))\cap W^{1,\infty}(0,T;L^2(\Omega;\R^3)).
\end{align}
Furthermore, from the fact that $p>2$, and Lemma \ref{lemmacomp0} we known that the Banach space $\vc V\cap W^{1,p}(\Omega;\R^3)$ is compact in $H^1(\Omega;\R^3)$. Hence, by \eqref{3d conv1 end}--\eqref{3d conv4 end} and the Aubin-Lions lemma, $\vc y_N$ converges to $\vc y$ also in the sense of \eqref{3d conv5}. This, together with \eqref{energia inmezzo N3}, \eqref{N3 con P} and Vitali's theorem, lead
\beq
\label{ennsima forte}
\nabla\vc y_N\to  \nabla\vc y,\qquad\text{strongly in $L^{\bar q}(\Omega\times(0,T))$},
\eeq
for every $\bar q < p$. We can thus pass again to the limit in \eqref{energia inmezzo N2}, and deduce 
\[
\begin{split}
\frac1k\|\Delta\vc y\|_{L^\infty(0,T;L^2)}^2
+\|\dot{\vc y}\|^2_{L^\infty(0,T;L^2)}
&+\|{\theta}\|^2_{L^\infty(0,T;L^2)}+\|\nabla{\theta}\|^2_{L^2(0,T;L^2)}
+\esssup_{t\in(0,T)} k \mint \phi_2(\nabla\vc y(t))\,\mathrm d\vc x
\\
&\leq c\Bigl(C^* +\frac1k\|\Delta\vc y_a\|^2_2
+ 3k \mint \phi_2(\nabla \vc y_a)\,\mathrm d\vc x\Bigr),
\end{split}
\] 
a.e. $t\in(0,T)$. Now the Poincar\'e inequality, together with \eqref{altrodasootyn} lead to
$$
\mint \phi_2(\nabla\vc y)\,\mathrm{d}\vc x\geq \|\nabla\vc y\|_p^p-c\geq c_0 \|\vc y\|_{W^{1,p}}^p-c -c\bigg|\mint \vc y\,\mathrm{d}\vc x \bigg|^p\geq c_0 \|\vc y\|_{W^{1,p}}^p-C^*_T - c\bigg|\mint \vc y_a\,\mathrm{d}\vc x \bigg| ^p,
$$
where we made use of \eqref{bound media S} in the last inequality. Combining the last two estimates we thus deduce \eqref{che bund 3d}. Furthermore, by arguing as above, we can use \eqref{stime crescita} and \eqref{ennsima forte} to pass to the limit as $N\to\infty$ in \eqref{approx 1a quater}--\eqref{approx 1b quater} and deduce \eqref{equazione 1a}--\eqref{equazione 1b}. Here, the only additional difficulty is to pass to the limit in the term $(\vc y_N,\boldsymbol\psi)_{H^2}$, which, thanks to \eqref{energia inmezzo N3} can be treated as follows
$$
\frac{1}{N}\bigl |(\vc y_N,\boldsymbol\psi)_{H^2} \bigr| \leq  \frac{1}{N}\|\vc y_N\|_{H^2}\|\boldsymbol\psi\|_{H^2}\leq \frac{\sqrt{c_{T,k}}}{\sqrt{N}}\|\boldsymbol\psi\|_{H^2}\to 0,\quad\text{as $N\to\infty$, a.e. $t\in(0,T)$.}
$$
Now, the existence result for a generic initial datum $\vc y_a\in \vc V\cap W^{1,p}(\Omega;\R^3)$ can be recovered as  follows: we take a sequence of $\vc y_a^{(j)}\in H^2(\Omega;\R^3)\cap W^{1,p}(\Omega;\R^3)$ converging strongly in $\vc V\cap W^{1,p}(\Omega;\R^3)$ to $\vc y_a$. In this case, 
\beq\label{con j}
\mint \phi_2(\nabla \vc y_{a}^{(j)})\leq c,\qquad \mint \phi_2(\nabla \vc y_{a}^{(j)})\to\mint \phi_2(\nabla \vc y_{a}),
\eeq
so thanks to \eqref{che bund 3d}, the sequence of solutions $(\theta_j,\vc y_j)$ related to $\vc y_a^{(j)}$ satisfies
\beq
\label{energia inmezzo N4}
\begin{split}
\Bigl(\|\Delta\vc y_j\|_2^2+\|\dot{\vc y}_j\|_2^2+\|\nabla\vc y_j\|_{W^{1,p}}^p+\|{\theta}_j\|^2_2\Bigr)(t)+{c_m}\tint\|\nabla{\theta}_j\|^2_2\,\mathrm d\tau
\leq c_{T,k}
,\qquad\text{a.e. $t\in(0,T)$}.
 \end{split} 
\eeq
Therefore we can deduce the existence of 
$$
\vc y \in L^\infty(0,T;\vc V\cap W^{1,p}(\Omega;\R^3))\cap W^{1,\infty}(0,T;L^2(\Omega;\R^3)),\qquad
\theta \in L^2(0,T;H^1(\Omega))\cap L^\infty(0,T;L^2(\Omega)),
$$
such that, up to a non relabelled subsequence, $(\theta_j,\vc y_j)$ converge to $(\theta,\vc y)$ in the sense of \eqref{3d conv1 end}--\eqref{ennsima forte}. Passage to the limit to recover that $(\theta,\vc y)$ satisfy \eqref{che bund 3d}, \eqref{equazione 1a}--\eqref{equazione 1b} follows the steps above.\\

Finally, by comparison in \eqref{equazione 1b} we also know that $$\ddot{\vc y}\in L^2(0,T;(W^{1,\tilde p}(\Omega;\R^3)\cap \vc V)^*
).$$ Hence, by \cite[Prop. II.5.11]{Boyer} and \cite[Lemma II.5.9]{Boyer} we deduce $\dot{\vc y}\in C([0,T]; L^2_w(\Omega;\R^3))$. On the other hand, \eqref{equazione 1a} implies also $\dot{\theta}+\dot\eta_1(\nabla\vc y) \in L^2(0,T;H^{-1}_D(\Omega))$. As by \eqref{stime crescita} and $\vc y\in L^\infty(0,T;W^{1,p}(\Omega;\R^3))$,
$$
\esssup_{t\in(0,T)}\mint\big|\eta(\nabla\vc y_n(t))\big|^{\frac p {r}}\,\mathrm d\vc x \leq c + \esssup_{t\in(0,T)}\mint\big|\nabla\vc y_n\big|^p\,\mathrm d\vc x\leq c_{T,k},
$$
it must hold
$\eta_1(\nabla\vc y) \in L^\infty(0,T;L^{p/{r}}(\Omega))$. Therefore, again by \cite[Prop. II.5.11]{Boyer} and \cite[Lemma II.5.9]{Boyer} we thus deduce that $$\theta+\eta_1(\nabla\vc y)\in C([0,T]; L^{\min{\{2,p/{r}}\}}_w(\Omega)).$$
\end{proof}

\section{Convergence as $k\to\infty$ to a limiting constrained theory}
\label{inelali}
{
As mentioned above, in this section we send $k$ to $\infty$ in \eqref{equazione 1a}--\eqref{equazione 1b}, and obtain in the limit a constrained theory for the deformation gradient $\nabla\vc y$. This is equivalent to assuming that elastic constants tend to infinity, which, as remarked in \cite{BallChuJames}, is usually a reasonable approximation when studying martensitic phase transitions with no external (or at least small) load. In this way we capture the essential behaviour of a generic free energy satisfying the properties listed in Section \ref{nonlonear elastic}, and neglect at the same time all aspects depending on the growth of the energy density. 
In the limit, the weak formulation of the energy conservation equation \eqref{equazione 1a} becomes a heat equation with a heat source at the austenite-martensite phase interface. This heat source is proportional to the difference of entropy between martensite and austenite, that is to the latent heat, is concentrated at the phase boundary, which might be sharp or diffuse, and depends also on the velocity of the phase transition. Where the transition is from austenite to martensite, the heat source has a positive sign, while its sign is negative where the transition is from martensite to austenite. The deformation gradient $\nabla\vc y_k$ generates as $k$ tends to infinity a Gradient Young Measure supported on $K\cup SO(3)$. As shown in Remark \ref{equaz 1b limit} the evolution of the obtained Gradient Young Measure cannot be deduced from the conservation of momentum alone.\\
}

In what follows we will make use of Young Measures, for which we refer the reader to \cite{BallYM,BallKoumatosQC,Pedregal} and references therein. Below, the space of Young Measures is denoted by $L^\infty_{w*}(\Omega\times(0,T);\mathcal M_1(\R^{3\times 3}))$, where $\mathcal M_1(\R^{3\times 3})$ is the space of probability measures on $\R^{3\times 3}$ and where the subscript $w^*$ stands for the fact that we endow this space with the weak$*$ topology. 
\begin{proposition}
\label{prop limit model}
Let 
$T>0$, let $(\theta_k, \vc y_k)$ be a solution of \eqref{equazione 1a}--\eqref{equazione 1b} given by Theorem \ref{esiste 3d} and such that $\vc y|_{t=0}=\vc y_{a,k}\in W^{1, p}(\Omega,\R^3)\cap \vc V$. 
Assume also $\vc y_{a,k}$ to be such that 
$$\mint \phi_2(\nabla\vc y_{a,k})\leq c k^{-1},\qquad\bigg|\mint \vc y_{a,k}\,\mathrm{d}\vc x \bigg|\leq c,\qquad \|\Delta\vc y_{a,k}\|_2\leq ck
$$ 
for every $k\geq1$, and for some positive constant $c$ independent of $k$. Then there exist 
\beq
\label{funzlimit 0}
\begin{split}
&\vc y\in L^\infty(0,T;W^{1,p}(\Omega;\R^3))\cap W^{1,\infty}(0,T;L^2(\Omega;\R^3)),\\
&\theta\in L^2(0,T;H^1(\Omega))\cap L^\infty(0,T;L^2(\Omega)),
\end{split}
\eeq
and $\nu_{\vc x,t}\in L^\infty_{w^*}(\Omega\times(0,T);\mathcal{M}_1(\R^{3\times3})),\nu_{\vc x,0}\in L^\infty_{w^*}(\Omega;\mathcal{M}_1(\R^{3\times3}))$ such that, up to a subsequence, 
\beq
\label{refff}
\begin{split}
&\vc y_k \tow \vc y,\qquad\text{weakly$*$ in }L^\infty(0,T;W^{1,p}(\Omega;\R^{3\times3}))
,\\
&\theta_k \tow \theta,\qquad\text{weakly in }L^2(0,T;H^{1}(\Omega))
,\\
&\delta_{\nabla\vc y_k(\vc x,t)} \tow \nu_{\vc x,t},\qquad\text{weakly$*$ in }L^\infty_{w^*}(\Omega\times(0,T);\mathcal{M}_1(\R^{3\times3})),\\
&\delta_{\nabla\vc y_{a,k}(\vc x)} \tow \nu_{\vc x,0},\qquad\text{weakly$*$ in }L^\infty_{w^*}(\Omega;\mathcal{M}_1(\R^{3\times3})).
\end{split}
\eeq
Furthermore, $\nu_{\vc x,t},\nu_{\vc x,0}$ satisfy
\beq
\label{supporto phi2}
\begin{split}
&\supp \nu_{\vc x,t} \subset K\cup SO(3),\qquad \text{for a.e. $(\vc x,t) \in\Omega\times(0,T)$},\\
&\supp \nu_{\vc x,0} \subset K\cup SO(3),\qquad \text{for a.e. $\vc x \in\Omega$},
\end{split}
\eeq
and $(\theta,\nabla \vc y,\nu_{\vc x,t})$ are satisfying the following weak formulation of the equation governing the conservation of energy
\beq
\label{cons ener limit}
\begin{split}
\Tint\mint \zeta\mt C\nabla\theta\cdot\nabla\omega \,\mathrm{d}\vc x = \Tint\mint \biggl(\theta_T\int_{\R^{3\times3}} \eta_1(\mt A)\,\mathrm{d}\nu_{\vc x,t}(\mt A)+\gamma\rho_0\theta\biggr)\dot{\zeta}\omega \,\mathrm d\vc x\mathrm dt\\
+\zeta(0)\mint\omega\biggl(\theta_T\int_{\R^{3\times3}} \eta_1(\mt A)\,\mathrm{d}\nu_{\vc x,0}(\mt A)+\gamma\rho_0\theta_0\biggr)\,\mathrm d\vc x,
\end{split}
\eeq
for each $\zeta\in C^1_c(-T,T)$, $\omega\in C^1(\Omega)\cap H^1_D $, and
\beq
\label{equazione 1c}
\int_{\R^{3\times3}} \mt A \,\mathrm{d}\nu_{\vc x,t}(\mt A) = \nabla\vc y (\vc x, t),\qquad\text{a.e. $(\vc x,t)\in \Omega\times(0,T)$}.
\eeq 
Also $\theta=\theta_B$ on $\partial\Omega_D$ for a.e. $t\in(0,T)$ and $\dot\theta+\frac{\partial}{\partial t}\int_{\R^{3\times3}}\eta(\mt A)\nu_{\vc x,t}(\mt A)\in L^2(0,T;H^{-1}_D)$.
\end{proposition}

\begin{proof}
Thanks to the hypotheses on $\vc y_{a,k}$ we can deduce the existence of a positive constant $c_T$ such that \eqref{che bund 3d} becomes
$$
\|\dot{\vc y}_k\|_{2}^2(t)+\|\theta_k\|^2_{2}(t)+\tint\|\nabla{\theta}_k\|^2_{2}(\tau)\,\mathrm{d}\tau + \|\vc{y}_k\|^p_{W^{1,p}}(t)+\frac1k\|\Delta\vc y_k\|_{2}^2(t)\\
 + k \mint \phi_2(\nabla \vc y_k(\vc x,t))\,\mathrm d\vc x \leq c_T,
$$ 
for a.e. $t\in (0,T)$.
Therefore, 
we can deduce the existence of $(\vc y,\theta)$ satisfying \eqref{funzlimit 0}, and of a converging subsequence satisfying \eqref{refff}\textsubscript{1}--\eqref{refff}\textsubscript{2}. 
\\

Now, for fixed $\eps>0$ we have that the above inequality implies 
$$
\frac{c_TT}k\geq \Tint\int_\Omega \phi_2(\nabla \vc y_k)\,\mathrm{d}\vc x\,\mathrm{d}t
\geq c_\eps \mathscr L^4\bigl\{(\vc x,t)\in\Omega\times(0,T)\colon |\nabla \vc y_k(\vc x,t) - \mt F|\geq \eps,\,\forall \mt F\in K\cup SO(3) \bigr\} ,
$$
for some constant $c_\eps>0$ depending on $\eps$ and on the continuous function $\phi_2$. Therefore,
$$
\lim_{k\to\infty}\mathscr L^4\bigl\{(\vc x,t)\in\Omega\times(0,T)\colon |\nabla \vc y_k(\vc x,t) - \mt F|\geq \eps,\,\forall \mt F\in K\cup SO(3) \bigr\}
= 0,
$$
which is convergence in measure of $\nabla \vc y_k$ to $K\cup SO(3)$ on $\Omega\times(0,T)$. Thus, the fundamental theorem of Young measures (see e.g., \cite{BallYM}) assures the existence of a non-relabelled subsequence and of a family of parametrized probability measures $\nu_{\vc x,t}\in L_{w^*}^\infty((0,T)\times\Omega;\mathcal M_1(\R^{3\times 3}))$ such that
$$
\supp \nu_{\vc x,t} \subset K\cup SO(3),\quad \text{ for a.e. 
$(\vc x,t)\in\Omega\times(0,T)$},
$$
and
$$
\lim_n\Tint\mint h(\nabla\vc y_k(\vc x,t))\zeta(\vc x,t)\,\mathrm{d}\vc x\,\mathrm{d} t=
\Tint\mint \int_{\R^{3\times 3}}h(\mt A)\,\mathrm d\nu_{\vc x,t}(A)\zeta(\vc x,t)\,\mathrm{d}\vc x\,\mathrm{d} t,
$$
for every $h\colon \R^{3\times 3}\to \R$ of growth strictly less than $p$, and every $\zeta\in L^\infty((0,T)\times\Omega).$ Choosing $h(\mt A)=\mt A$, by \eqref{3d conv4} we immediately get that \eqref{equazione 1c} holds.
Furthermore,
\begin{align*}
\eta_1(\nabla \vc y_k) \tow \int_{\R^{3\times3}}\eta_1(\mt A)\,\mathrm{d}\nu_{\vc x,t}(\mt A),\quad\text{ in $L^q(\Omega\times(0,T))$
},
\end{align*}
for every $q\in[1,\frac p{r})$. A similar argument entails that $\vc y_{a,k}$ generates $\nu_{\vc x,0}$ and that 
\begin{align*}
\supp \nu_{\vc x,0} \subset K\cup SO(3),\quad \text{ for a.e. 
$\vc x\in\Omega$},\\
\eta_1(\nabla \vc y_{a,k}) \tow \int_{\R^{3\times3}}\eta_1(\mt A)\,\mathrm{d}\nu_{\vc x,0}(\mt A),\quad\text{ in $L^q(\Omega)$
},
\end{align*}
for every $q\in[1,\frac p{r})$.\\

We can thus pass to the limit in \eqref{equazione 1a} and get \eqref{cons ener limit}. Finally, by comparison, we deduce $\dot\theta+\frac{\partial}{\partial t}\int_{\R^{3\times3}}\eta(\mt A)\nu_{\vc x,t}(\mt A)\in L^2(0,T;H^{-1}_D)$.
\end{proof}

\begin{remark}
\label{equaz 1b limit}
\normalfont
Under the assumptions of Proposition \ref{prop limit model}, when passing to the limit as $k\to\infty$, we notice that the equation for the conservation of momentum, that is \eqref{equazione 1b}, reduces to
\beq
\label{limite y eq}
\mint \int_{\R^{3\times 3}}\frac{\partial\phi_{2}}{\partial\mt F}(\mt A)\,\mathrm d\nu_{\vc x,t}(\mt A)\,\mathrm d\vc x : \nabla\boldsymbol\psi \,\mathrm d\vc x = 0,\qquad \forall \boldsymbol\psi\in C^1(\Omega;\R^3),\qquad\text{a.e. $t\in (0,T)$}.
\eeq
Indeed, \eqref{equazione 1b} with $\xi\in C^1_c(0,T)$ becomes
\beq
\label{limite y eq 2}
\begin{split}
\Tint\mint& \xi\frac{\partial\phi_{2}}{\partial\mt F}(\nabla\vc y_k)\,\mathrm d\vc x : \nabla\boldsymbol\psi \,\mathrm d\vc x\,\mathrm d t \\ 
&=-\frac 1k \Tint\mint\biggl (\xi\frac{\partial\phi_{1}}{\partial\mt F}(\nabla\vc y_k,\theta): \nabla\boldsymbol\psi -\rho_0\dot{\xi}\dot{\vc y}_k\cdot\boldsymbol\psi+\frac1k\Delta\vc y_k\cdot\Delta\boldsymbol\psi \biggr)\,\mathrm d\vc x\,\mathrm d t.
\end{split}
\eeq
By the Cauchy-Schwarz and the H\"older inequalities, thanks to \eqref{stime crescita} we get
\[
\begin{split}
\Tint\mint\Bigl(&\xi\frac{\partial\phi_{1}}{\partial\mt F}(\nabla\vc y_k,\theta): \nabla\boldsymbol\psi\Bigr)\,\mathrm d\vc x\,\mathrm d t \\
&\leq\hat c_T\bigl(1+\|\nabla\vc y_k\|_{L^\infty(0,T;L^p)
}^p+\|\nabla\vc y_k\|_{L^\infty(0,T;L^p)
}^{\tilde r}\|\theta_k\|_{L^2(0,T;L^6)}
\bigr)\|\xi\|_{C^1}\|\boldsymbol\psi\|_{C^1}\\
&\leq \hat C_T\|\xi\|_{C^1}\|\boldsymbol\psi\|_{C^1},
\end{split}
\]
where $c_T,\hat C_T$ are positive constants independent of $k$, and where we also made use of the bounds on $(\theta_k,\nabla\vc y_k)$ in the proof of Proposition \ref{prop limit model}. On the other hand, 
\begin{align*}
\Tint\mint\rho_0\dot{\xi}\dot{\vc y}_k\cdot\boldsymbol \psi\,\mathrm d\vc x\,\mathrm d t 
\leq c_T\rho_0 \| \dot{\vc y}_k\|_{L^\infty(0,T;L^2)}
 \|\boldsymbol\psi\|_2\|\xi\|_{C^1}
\leq \tilde c_T\|\xi\|_{C^1}\|\boldsymbol\psi\|_{C^1},\\
\Tint\mint\frac1k \xi\Delta\vc y_k\cdot \Delta\boldsymbol \psi\,\mathrm d\vc x\,\mathrm d t \leq \frac{c_T}k\|\Delta\vc y_k\|_{L^\infty(0,T;L^2)}\|\xi\|_{C^1}\|\boldsymbol\psi\|_{H^2}\leq \tilde c_T\|\xi\|_{C^1}\|\boldsymbol\psi\|_{H^2}
\end{align*}
for some positive constant $\tilde c_T$ independent of $k$. A passage to the limit in \eqref{limite y eq 2}, together with \eqref{stime crescita} and the fact that $\nabla\vc y_k$ generates $\nu_{\vc x,t}$ (cf. \eqref{refff}) lead to \eqref{limite y eq}. 
In conclusion, \eqref{cons ener}\textsubscript{2} does not determine the evolution for $\nu_{\vc x,t}$ after taking the limit $k\to\infty$. Actually, given that $\phi_2$ is smooth, and hence $\frac{\partial\phi_2}{\partial\mt F}(\mt F)= \mt 0$ for each $\mt F\in SO(3)\cup K$, \eqref{limite y eq} does not even add any further information to \eqref{supporto phi2}. 
\end{remark}
\begin{remark}
\label{tempus reg}
\normalfont
After taking the limit $k\to\infty$, the only information on the time evolution of $\vc y$ is embedded in $\frac{\partial}{\partial t}\int_{\R^{3\times3}}\eta_1(\mt A)\,\mathrm{d}\nu_{\vc x,t},$ and in $\vc y \in W^{1,\infty}(0,T;L^{2}(\Omega))\cap L^{\infty}(0,T;W^{1,\infty}(\Omega))$. Indeed, we point out that, by \eqref{supporto phi2}, $\int_{\R^{3\times3}}\eta_1(\mt A)\,\mathrm{d}\nu_{\vc x,t}\in L^\infty(\Omega\times(0,T))$. Thus, $$\int_{\R^{3\times3}}\eta_1(\mt A)\,\mathrm{d}\nu_{\vc x,t}+{\theta}\in H^1(0,T;H^{-1}_D)\cap L^\infty((0,T);L^2(\Omega)).$$ 
By \cite[Lemma II.5.9]{Boyer} this also belongs to $C(0,T;L^2_w(\Omega))$, from which we can make sense of the initial conditions.
Furthermore, we also have that \cite[Lemma II.5.9]{Boyer} together with \eqref{funzlimit 0} imply that $\vc y\in C(0,T;W^{1,\infty}_w(\Omega))$. 
\end{remark}
\begin{remark}
\rm
The gradient Young measure generated by $\nabla\vc y_k$ as $k\to\infty$ is in general non-trivial, that is not of the form $\delta_{\nabla\vc y(\vc x,t)}$ for a.e. $(\vc x,t)\in\Omega\times(0,T)$. Indeed, we know from the static theory (see Section \ref{nonlonear elastic}) that martensitic microstructures may arise in order to achieve compatibility between phases, and hence, as $k\to\infty$ faster and faster oscillations may occur in $\nabla\vc y_k$, generating non-trivial gradient Young measures.
\end{remark}

\begin{remark}
\label{Remark inter}
\rm
If we assume as in \cite{FDP2} or in Section \ref{PDE theory} below that there exist $\Omega_A(t),\Omega_M(t)\subset\Omega$ open such that
\beq
\label{fivestar}
\Omega_A(t)\cap\Omega_M(t)=\varnothing,
\qquad \mathscr{L}^3 \bigl( \Omega\setminus(\Omega_A(t)\cup\Omega_M(t))\bigr)=0, \quad\text{a.e. $t\in(0,T)$},
\eeq
and
\begin{align*}
\nu_{\vc x,t}(SO(3))= 1,\qquad\text{a.e. $\vc x\in\Omega_A(t)$, a.e. $t\in(0,T)$}, \\
\nu_{\vc x,t}(K)= 1,\qquad\text{a.e. $\vc x\in\Omega_M(t)$, a.e. $t\in(0,T)$}.
\end{align*} then $$\int_{\R^{3\times3}}\eta_1(\mt A)\,\mathrm{d}\nu_{\vc x,t}=-\alpha\chi_{\Omega_m},$$ where
by $\chi_{\Omega_m}$ we denoted the indicator function on $\Omega_M$. The formula for differentiation of integrals on time dependent domains implies
\beq
\label{derivata materiale}
\langle \dot\chi_{\Omega_M}(\nabla \vc y),\psi\rangle = \dert \int_{\Omega_M} \psi \,\mathrm d\vc x = \int_{\Gamma(t)} (\vc v \cdot \vc n) \psi\,\mathrm d\mathscr H^2,\qquad\forall \psi \in C^\infty_0(\Omega),
\eeq
provided $\Omega_A(t),\Omega_M(t)$ 
and $\vc v\cdot \vc n$ are smooth enough (see e.g., \cite{Flanders}). 
Here $\Gamma(t):=\Omega\setminus(\Omega_A\cup\Omega_M)(t)$ is a surface separating $\Omega_A(t)$ from $\Omega_M(t)$, $\vc n$ denotes the outer normal to $\Omega_M$ and $\vc v(\vc s)$ is the velocity of the interface at the point $\vc s\in \Gamma(t)$ at time $t$. By $\langle \cdot,\cdot\rangle$ we denoted the duality pairing between a distribution and a test function. A version of \eqref{derivata materiale} under the moving mask approximation can be found in \cite[Corollary 4.4]{FDP2}. It thus appears clear from \eqref{derivata materiale} that the model needs to be closed with some relation between $\vc v$ and $\nabla \vc y$, $\theta_B$, and possibly $\theta$, as much as it requires an initial condition for $\Omega_M(t)$. 
\end{remark}

\section{Connections with the moving mask assumption}
\label{connection}
In this section we draw connections between the results of Section \ref{inelali}, and the moving mask assumptions introduced in \cite{FDP2}. 
Let us start by recalling the \textit{moving mask} assumptions:
\begin{enumerate}[label=(MM\arabic*)]
\setlength\itemsep{0.1em}
\item\label{MM1} the phases are separated, that is at almost every point we have either austenite or martensite but not both, so that phase interfaces between austenite and martensite are sharp (cf. \eqref{fivestar});
\item\label{MM2} during the phase transition, the deformation gradient remains equal to the identity in the austenite region. This is the case, for example, when the austenite region is connected;
\item\label{MM4} microstructures do not change after the transformation has happened;
\item\label{MM3} the phase interface moves continuously. More precisely, for almost every point $\vc x$ in the domain, there exists a time when $\vc x$ is contained in the phase interface. 
\end{enumerate}
As proved in Section \ref{inelali}, in the limit as $k$ tends to $\infty$, the system \eqref{cons ener} reduces to finding a function $\theta$ and a time-dependant gradient Young measure $\nu_{\vc x,t}$ satisfying 
\beq
\label{supp strong}
\supp \nu_{\vc x,t} \subset K\cup SO(3),\qquad\text{a.e. $(\vc x,t) \in \Omega\times(0,T)$},
\eeq
and the weak formulation in \eqref{cons ener limit} of
\beq
\label{heat eq strong}
\gamma\rho_0\theta_t - \Div(\mt C\nabla\theta) = -\theta_T\frac{\partial}{\partial t}\int_{\R^{3\times3}}\eta_1(\mt A)\,\mathrm{d}\nu_{\vc x,t}(\mt A).
\eeq
On the one hand, it is clear that, given an arbitrary parametrised family of measures $\nu_{\vc x,t}\in L^\infty(\Omega\times(0,T);\mathcal{M}_1(\R^{3\times3}))$ satisfying \eqref{supporto phi2}, \eqref{equazione 1c}, $\frac{\partial}{\partial t}\int_{\R^{3\times3}}\eta_1(\mt A)\,\mathrm{d}\nu_{\vc x,t}(\mt A)\in L^2(0,T;H_D^{-1})$ and \ref{MM1}--\ref{MM3}, then one can find the related temperature field $\theta(\vc x,t)\in L^2(0,T;H^1_D)$ by solving the weak form of the heat equation \eqref{heat eq strong}. This means that sufficiently regular moving mask solutions are solutions to \eqref{supp strong}--\eqref{heat eq strong}.\\
On the other hand, as pointed out in Remark \ref{Remark inter}, the system \eqref{supp strong}--\eqref{heat eq strong} is underdetermined, and should be closed with some constitutive relation between the velocity of the phase interface and the temperature, $\nu_{\vc x,t}$, and possibly other quantities. This is done, for example, in Section \ref{PDE theory} in a one-dimensional context. 
However, we conjecture that \ref{MM1}--\ref{MM3} can be deduced to hold in the limit as $k\to\infty$ for $(\theta_k,\vc y_k)$, weak solutions to \eqref{Pb bc}--\eqref{cons ener} provided by Theorem \ref{esiste 3d}.  \\

In order to justify \ref{MM1} one should prove that ,
$$
\mathscr L^4\bigl(\bigl\{(\vc x,t)\colon \nu_{\vc x,t}(K) \in \{0,1\}	\bigr\}\bigr) =0,
$$
where $\nu_{\vc x,t}$ is as in Proposition \ref{prop limit model}. We recall that (see e.g., \cite{BallYM,Muller,Pedregal})
$$
\nu_{\vc x,t}(\mathcal B) = \lim_{\delta\to 0}\lim_{k\to\infty}\frac{\mathscr L^4\bigl(\bigl\{ (\vc z,s) \in B((\vc x,t);\delta)\colon \nabla \vc y_k(z,s)\in \mathcal B \bigr\}\bigr)}
{\mathscr L^4\bigl(B((\vc x,t);\delta)\bigr)},
$$
for every $\mathcal B\subset \R^{3\times 3}$ Borel, and where $B((\vc x,t);\delta)$ is the open ball of radius $\delta$ centred at $(\vc x,t)$. Here, $\nabla\vc y_k$ is the sequence generating $\nu_{\vc x,t}$ as in \eqref{refff}. Therefore, to rigorously prove \ref{MM1}, one could prove that for almost every $(\vc x,t)\in \Omega\times(0,T),$ and for every $\eps>0$ there exist $\delta_0>0$ such that for every $\delta\in(0,\delta_0)$
$$
\lim_{k\to\infty}\frac{\mathscr L^4\bigl(\bigl\{ (\vc z,s) \in B((\vc x,t);\delta)\colon \nabla \vc y_k(z,s)\in \mathcal B \bigr\}\bigr)}
{\mathscr L^4\bigl(B((\vc x,t);\delta)\bigr)} \in (0,\eps)\cup(1-\eps,1+\eps).
$$ 
This result should be easier to prove in the case of materials satisfying the cofactor conditions where martensitic laminates can form exact austenite-martensite interfaces (see \cite{JamesHyst,FDP3}). Furthermore, one should restrict to initial data $\vc y_{a,k}$ generating a gradient Young measure $\nu_{\vc x,0}$ such that $\nu_{\vc x,0}(K)\in\{0,1\}$ for a.e. $\vc x\in \Omega$.\\
Regarding \ref{MM2}, it is a consequence of a well known result by Reshetnyak (see \cite{Res}, but also \cite{BallJames2,Muller}) that a parametrised measure $\mu_{\vc x}$ with $\supp \mu_{\vc x} \subset SO(3)$ for a.e. $\vc x$ in an open connected subset $\mathcal{A}\subset\Omega$ must be of the form $\mu_{\vc x}=\delta_{\mt R}$ for a.e. $\vc x\in \mathcal{A},$ some $\mt R\in SO(3)$, and where $\delta_{\mt R}$ denotes a Dirac delta at $\mt R.$ So the problem reduces to show that the set $\{\vc x\colon \nu_{\vc x,t}(SO(3))=1\}$ is connected, or, more in general, indecomposable (see \cite{MullerDolzman,Muller}).\\
A context where it might be easier to show that \ref{MM1}--\ref{MM2} hold is when $|\det(\mt F)-1|>\delta$ for every $\mt F\in K$, and for some $\delta>0$. Indeed, in this case, we expect for $k$ large enough to just have nucleations at the corners of the sample (see \cite{BallKoumatosQC,BattSme}), and the presence of martensitic islands in austenite (or viceversa) to be energetically too expensive.\\
Proving \ref{MM4} rigorously from the solutions to \eqref{Pb bc}--\eqref{cons ener} in the limit as $k$ tends to infinity might be very difficult in general. Assuming for simplicity \ref{MM1}, we now give an intuitive explaination of why \ref{MM4} should hold based on rigidity, which is well known to play an important role in vectorial problems. As the map $\vc y\in W^{1,\infty}(\Omega;\R^3)$, with $\nabla\vc y$ satisfying \eqref{equazione 1c}, is Lipschitz continous, its gradient cannot be arbitrary (see e.g., \cite{BallJames1,BallJames2,Muller}). {Therefore, changing the macroscopic deformation gradient $\nabla\vc y$ in a subset of $\{\vc x\colon \nu_{\vc x,t}(K)=1\}$ of positive measure, might require changing $\nabla\vc y$ in a much larger set, leading to discontinuities in $t$ of $\vc y$ on a subset $\Omega$ of positive measure. But such a jump in $t$ would contradict the fact that, by Proposition \ref{prop limit model}, $\vc y\in W^{1,\infty}(0,T;L^2(\Omega;\R^3)).$\\
} 
%
%
The continuous movement of the phase interface, that is assumption \ref{MM3}, is partially a consequence of the observations in Remark \ref{tempus reg}. Indeed, let us assume that \ref{MM1} and \ref{MM4} hold, and, for simplicity, that 
$$
\min_{\mt F\in K^{qc}}\min_{\mt R\in SO(3)} |\mt F-\mt R| \geq \eps,
$$
for some $\eps>0.$ Here we are taking $SO(3)$ and not $SO(3)^{qc}$ as these two sets coincide (see e.g., \cite{BallJames2}). In this case, suppose \ref{MM3} is not satisfied. Then there exist $t_0\in(0,T)$, and a measurable set $\mathcal{A}$, with $\meas (\mathcal{A})>0$, such that the macroscopic deformation gradient $\nabla \vc y$ satisfies $\nabla \vc y(\vc x,s)\in SO(3)$ for a.e. $(\vc x,s)\in\Omega\times(0,t_0)$, and $\nabla \vc y(\vc x,s)\in K$ for a.e. $(\vc x,s)\in\Omega\times[t_0,t_0+\delta)$, for some $\delta>0$. Given the fact that we are assuming that \ref{MM4} also holds, there exist $\boldsymbol\psi\in W^{1,\infty}(\Omega;\R^3)$, independent of time, such that $\nabla\boldsymbol\psi = \nabla\vc y(\vc x,t)-\nabla\vc y(\vc x,s)$, for every $s,t$ satisfying $s<t_0<t<t_0+\delta$. In this case, 
$$
\mint (\nabla\vc y(\vc x,t)-\nabla\vc y(\vc x,s)):\nabla\psi \,\mathrm{d}\vc x = \mint |\nabla\vc y(\vc x,t)-\nabla\vc y(\vc x,s)|^2\,\mathrm{d}\vc x \geq \eps^2\meas(\mathcal{A})>0, 
$$ 
for every $s<t$, thus contradicting the fact that, by Remark \ref{tempus reg}, $\vc y\in C([0,T];W^{1,\infty}_{w^*}(\Omega;\R^3))$, which implies
$$
\mint (\nabla\vc y(\vc x,t)-\nabla\vc y(\vc x,s)):\nabla\psi \,\mathrm{d}\vc x\to 0, \qquad\text{as $s\to t$.}
$$

\section{The evolution of a simple laminate in a one-dimensional context}
\label{PDE theory}
{The aim of this section is to introduce some hypotheses that allow us to construct solutions to the limit problem \eqref{supporto phi2}--\eqref{cons ener limit} by solving a one-dimensional heat equation with a measure valued source depending nonlinearly on the unknown. This models the evolution of the phase boundary between a laminate of martensite and austenite and is a simplified version of the model in \cite{AberKnowles2,AberKim}. The resulting solutions are an example of solutions to \eqref{supp strong}--\eqref{heat eq strong} satisfying \ref{MM1}-\ref{MM3}.
}
\subsection{An example: the simple laminate}
\label{An ex easy}
Let us suppose that an austenite to martensite phase transition takes place in a parallelepiped with a circular base, that is, in cylindrical coordinates, $\Omega_C:=(0,2\pi)\times(0,r)\times(0,L)$ for some $0<r\ll L$. Suppose further that there exists a single phase interface $\Gamma(t)=\{\vc x\in\Omega_C\colon \vc x\cdot\vc e_3=u(t)\}$, $u(t)\in(0,L)$ perpendicular to $\vc e_3$ for every $t$, and that
\begin{align*}
\nu_{\vc x,t} = \lambda \delta_{\mt A}+(1-\lambda)\delta_{\mt B},\qquad\text{a.e. in $\Omega_M(t)$, a.e. $t>0$},\\
\nu_{\vc x,t} = \delta_{\mt 1},\qquad\text{a.e. in $\Omega_A(t)$, a.e. $t>0$},
\end{align*} 
with $\mt A,\mt B\in K$, $\lambda \in [0,1]$ fixed, and 
$$\Omega_M(t) = \{\vc x\in\Omega_C\colon \vc x\cdot\vc e_3<u(t)\},\qquad \Omega_A(t) = \{\vc x\in\Omega_C\colon \vc x\cdot\vc e_3>u(t)\}.$$
Here and below $\delta_{\mt F}$ is a Dirac measure at $\mt F$. 
By \eqref{equazione 1c}, the macroscopic deformation gradient, that is an average of the fine microstructures, is given by the barycentre of the Young measure, that is, in our case
\begin{align*}
\nabla \vc y = \int_{\R^{3\times3}}\mt F\,\mathrm{d}\nu_{\vc x,t}(\mt F) = \lambda \mt A+(1-\lambda)\mt B,\qquad\text{a.e. in $\Omega_M(t)$, a.e. $t>0$},\\
\nabla \vc y = \int_{\R^{3\times3}}\mt F\,\mathrm{d}\nu_{\vc x,t}(\mt F) = \mt 1,\qquad\text{a.e. in $\Omega_A(t)$, a.e. $t>0$}.
\end{align*}
However, as shown in \cite{BallJames1}, $\nabla\vc y$ is a gradient of a continuous deformation map $\vc y\colon \Omega_C\to \R^3$ if and only if there exist $\vc a\in\R^3$ 
such that
$$
\lambda \mt A+(1-\lambda)\mt B  = \mt 1 +\vc a\otimes \vc e_3.
$$
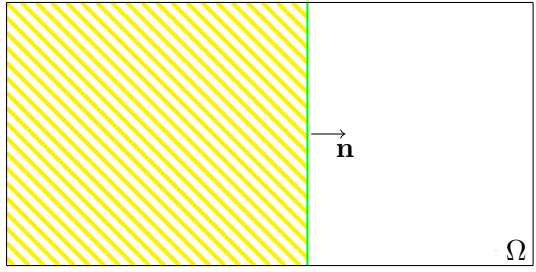
\begin{figure}
\centering
\begin{tikzpicture}
\draw[yellow, ultra thick] (0,0) -- (-3.5,3.5);
\draw[yellow, ultra thick] (-0.2,0) -- (-3.5,3.3);
\draw[yellow, ultra thick] (-0.4,0) -- (-3.5,3.1);
\draw[yellow, ultra thick] (-0.6,0) -- (-3.5,2.9);
\draw[yellow, ultra thick] (-0.8,0) -- (-3.5,2.7);
\draw[yellow, ultra thick] (-1.0,0) -- (-3.5,2.5);
\draw[yellow, ultra thick] (-1.2,0) -- (-3.5,2.3);
\draw[yellow, ultra thick] (-1.4,0) -- (-3.5,2.1);
\draw[yellow, ultra thick] (-1.6,0) -- (-3.5,1.9);
\draw[yellow, ultra thick] (-1.8,0) -- (-3.5,1.7);
\draw[yellow, ultra thick] (-2,0) --	(-3.5,1.5);
\draw[yellow, ultra thick] (-2.2,0) -- (-3.5,1.3);
\draw[yellow, ultra thick] (-2.4,0) -- (-3.5,1.1);
\draw[yellow, ultra thick] (-2.6,0) -- (-3.5,0.9);
\draw[yellow, ultra thick] (-2.8,0) -- (-3.5,0.7);
\draw[yellow, ultra thick] (-3,0) -- (-3.5,0.5);
\draw[yellow, ultra thick] (-3.2,0) -- (-3.5,0.3);
\draw[yellow, ultra thick] (-3.4,0) -- (-3.5,0.1);
\draw[yellow, ultra thick] (0.2,0) -- (-3.3,3.5);
\draw[yellow, ultra thick] (0.4,0) -- (-3.1,3.5);
\draw[yellow, ultra thick] (0.5,.1) -- (-2.9,3.5);
\draw[yellow, ultra thick] (0.5,.3) -- (-2.7,3.5);
\draw[yellow, ultra thick] (0.5,.5) -- (-2.5,3.5);
\draw[yellow, ultra thick] (0.5,.7) -- (-2.3,3.5);
\draw[yellow, ultra thick] (0.5,.9) -- (-2.1,3.5);
\draw[yellow, ultra thick] (0.5,1.1) -- (-1.9,3.5);
\draw[yellow, ultra thick] (0.5,1.3) -- (-1.7,3.5);
\draw[yellow, ultra thick] (0.5,1.5) -- (-1.5,3.5);
\draw[yellow, ultra thick] (0.5,1.7) -- (-1.3,3.5);
\draw[yellow, ultra thick] (0.5,1.9) -- (-1.1,3.5);
\draw[yellow, ultra thick] (0.5,2.1) -- (-0.9,3.5);
\draw[yellow, ultra thick] (0.5,2.3) -- (-0.7,3.5);
\draw[yellow, ultra thick] (0.5,2.5) -- (-0.5,3.5);
\draw[yellow, ultra thick] (0.5,2.7) -- (-0.3,3.5);
\draw[yellow, ultra thick] (0.5,2.9) -- (-0.1,3.5);
\draw[yellow, ultra thick] (0.5,3.1) -- (0.1,3.5);
\draw[yellow, ultra thick] (0.5,3.3) -- (0.3,3.5);
\draw[yellow, ultra thick] (0.5,3.5) -- (0.5,3.5);
\draw[green, thick] (0.5,0) -- (0.5,3.5);
\draw[thin] (-3.5,3.5) -- (3.5,3.5);
\draw[thin] (3.5,3.5) -- (3.5,0);
\draw[thin] (3.5,0) -- (-3.5,0);
\draw[thin] (-3.5,0) -- (-3.5,3.5);
\draw [->] (0.55,1.75) -- (1,1.75); 
\filldraw [red] (1,1.75) circle (0pt) node[anchor=north,black] {$\vc n $};

\filldraw [red] (3.,0.2) circle (0pt) node[anchor=west,black] {$\Omega$};

\end{tikzpicture}
\caption{Picture of a section of $\Omega_C$, on the left $\Omega_M(t)$ where $\nabla \vc y = \mt 1 +\vc a\otimes \vc n$, on the right $\Omega_A$, where $\nabla \vc y=\mt 1$ .}
\end{figure}
We can hence deduce that 
$$
\vc y=
\begin{cases}
\vc x+\vc a(\vc x\cdot \vc e_3)+\vc c_1(t),\qquad&\text{in $\Omega_M(t)$},\\
\vc x+\vc c_2(t),\qquad&\text{in $\Omega_A(t)$},
\end{cases}
$$
for some time dependent vectors $\vc c_1,\vc c_2\colon [0,\infty)\to\R^3$ such that $\vc c_1(t)-\vc c_2(t) = \vc a u(t)$. In an equivalent way, knowing $\vc c_1,\vc c_2$ we can determine the position of the phase boundary given by those $\vc x$ satisfying $(\vc x\cdot \vc e_3) = \frac{|\vc c_2(t)-\vc c_1(t)|}{|\vc a|}$. 
For simplicity we choose constant Dirichlet boundary conditions $\theta_B\neq\theta_T$ for $x_3=0,L$, and $\frac{\partial\theta}{\partial\vc m}=0$ on the other face, having normal $\vc m$. In this case, the equation governing the conservation of energy simplifies to a one-dimensional equation. We define $s\in(0,L)$ to be the coordinate in direction $\vc e_3$, we assume the phase interface to be located at $s=u_0$ when $t=0$, and assume without loss of generality that $\Omega_M(0)=\{s\colon s<u_0\}$. In this case, the classical formulation of \eqref{cons ener limit} reduces to 
\beq
\label{pb 1d simple}
\begin{cases}
\rho_0\gamma\dot\theta- K\theta_{ss}=\alpha
\dot\chi_{\Omega_M}(t),\qquad &\text{in $\Omega\times(0,T)$},\\
\chi_{\Omega_M}(s,t) =  \chi_{\{s<u(t)\}},\qquad &\text{in $(0,T)$},\\
\theta(0,t)=\theta(L,t)=\theta_B,\qquad &\text{in $(0,T)$},\\
\theta\bigl|_{t=0}=\theta_0,\qquad &\text{in $(0,L)$},\\
u(0) = u_0,
\end{cases}
\eeq
where we denoted by $\chi_{\mathcal B}$ the indicator function on the Borel set $\mathcal{B}$. Without loss of generality, below we assume $\alpha=\gamma=\rho_0=1$ to simplify notation. This system of equations describes the evolution of a sharp phase interface at $u(t)$. However, this is underdetermined, and needs to be closed with a constitutive relation for the interface speed $v=\dert \frac{|\vc c_2(t)-\vc c_1(t)|}{|\vc a|}=\dot{u}$ depending on the temperature at $u(t)$ and possibly other variables (see also Remark \ref{Remark inter}).
In order to respect the physics of the problem we want the martensite domain to expand when $\theta<\theta_T$ and to shrink when $\theta>\theta_T$. Mathematically, this means that we assume
\beq
\label{cond v}
\begin{cases}
v(\theta) >0,\qquad\text{if } \theta<\theta_T,\\
v(\theta) <0,\qquad\text{if } \theta>\theta_T.
\end{cases}
\eeq
Imposing 
\beq\label{chiusura}\dot u(t)=v(\theta(u(t),t))\eeq 
makes the problem a two-phase Stefan problem with a kinetic condition
at the free boundary (see e.g. \cite{Vis,VisLib,Xien}). The one-dimensional two-phase Stefan problem is usually closed by imposing that the temperature $\theta$ is equal to $\theta_T$ at the phase interface located in $s=u(t)$. However, in martensitic transformations nucleation usually happens strictly below the critical temperature due to hysteresis (see e.g., \cite{JamesMuller}). Given the high thermal conductivity of metals, the high-temperature phase, that is austenite, as much as the phase interface, can be at a temperature strictly lower than the critical-one. This is usually called undercooling in the literature of the Stefan problem (see e.g. \cite{VisLib}). Therefore, the condition $\theta=\theta_T$ at $s=u(t)$ is inaccurate, and is better replaced by \eqref{chiusura} (see also \cite{AberKnowles2,AberKim}). {As shown in \cite{Xien}, the classical two-phase Stefan problem can be recovered from the two-phase Stefan problem with a kinetic condition at the free boundary by assuming that $v$ is linear, that is $\dot u(t) = \frac1\eps(\theta_T-\theta(u(t),t))$, and passing to the limit as $\eps\to 0.$ This would correspond to a physical situation with no hysteresis and no undercooling.
}
Existence of a weak solution to system \eqref{pb 1d simple} has been studied in \cite{Vis}, while existence and uniqueness of suitably defined classical solutions is achieved in \cite{Xien}, under the assumption that $\vc v$ is linear. The aim of the next section is to show that, if $\min\{0,\theta_B\}\leq\theta_0(s)\leq \max\{0,\theta_B\}$ for a.e. $s\in(0,L)$, then the position of the phase interface $u$ is a monotone function of $t$, and reaches the domain boundary in a finite amount of time.



\subsection{Behaviour of solutions to the 1-dimensional problem}
For simplicity, we define a rescaled temperature $\bar\theta=\theta -\theta_B$, the rescaled critical temperature $\theta_c:=\theta_T-\theta_B$, and below we consider the following definition of weak solution
\begin{definition}
\label{Def ws}
Let $T>0$, and $\Omega:=(0,L)$. We say that $(\bar\theta,u)$ is a weak solution to problem \eqref{pb 1d simple} if
\begin{align*}
&u \in H^1(0,T),\\
&\bar\theta \in L^2(0,T; H^1_0(\Omega))\cap H^1(0,T;H^{-1}(\Omega)),\\
&\bt + \theta_B > 0,\qquad\text{ a.e. }x\in\Omega,\, t\in (0,T),
\end{align*}
and 
\begin{align}
\label{heat eq}
\langle \dot\bt,\psi\rangle + \mint K\bt_s\psi_s\,\mathrm{d}x = \langle \dot u \delta_{u} ,\psi\rangle ,\qquad  &\text{a.e. }t\in(0,T),\,\forall \psi \in H^1_{0}(\Omega),\\
\label{usimuove}
u(t)=u(0)+\tint \chi_{\{u(\tau)\in\Omega\}} v(\bt(u(\tau),\tau))\,\mathrm{d}\tau, \qquad&\forall t\in(0,T).
\end{align}
Furthermore, $\bar\theta(0,x)=\bart_0:=\theta_0-\theta_B$ 
almost everywhere in $\Omega$.
\end{definition}
\begin{remark}
\rm{
{
For the sake of coherence, we multiplied by $\chi_{\{t\colon u(t)\in(0,L)\}}$ the velocity $v$ of the interface, in order to avoid the non-physical situation where the interface moves out of the sample. This does not affect the $H^1(0,T)$ regularity of $u$.
%
%
}
}
\end{remark}
The following result concerning weak solutions holds 
\begin{theorem}
\label{esistenza 1d}
Let $v$ be Lipschitz, $\bart_0\in L^2(\Omega)$ and $u(0)\in\Omega$. Then there exists $(\bart,u)$ weak solution in the sense of Definition \ref{Def ws} to \eqref{pb 1d simple}.
\end{theorem}
We refer to \cite{Vis} for the proof of this result that can be carried out via a Galerkin approximation. 

The following proposition gives some good behaviour of the solutions dependent only on \eqref{cond v}:
\begin{proposition}
\label{MacPrin}
Let $(\bt,u)$ be a weak solution as in Definition \ref{Def ws}, $\theta_0$ be measurable
and let $v$ satisfy \eqref{cond v}. Then, if $\theta_c \leq 0$ and $\theta_c\leq\bart_0\leq 0$ a.e. in $\Omega$, we have
\beq
\label{limit 1}
\theta_c\leq\bt(s,t)\leq 0,\qquad\forall s\in\Omega,\text{ a.e. }
t\in (0,T).
\eeq
If $\theta_c \geq 0$ and $\theta_c\geq\theta_0\geq 0$ a.e. in $\Omega$, we have
\beq
\label{limit 2}
\theta_c\geq\bt(s,t)\geq 0,\qquad \forall s\in\Omega,\text{ a.e. }
t\in (0,T).
\eeq
\end{proposition}
\begin{proof}
The proof of a similar statement can be found in \cite{Vis}, but we report here the proof for the sake of completeness. We start with the proof of \eqref{limit 1}. Let us choose as a test function $\psi$ in \eqref{heat eq} $\psi=(\bt-\theta_c)^-:=- \max{(0,\theta_c-\bt)}$. In this case, by using \eqref{cond v} we have
$$
\frac12\dert \|(\bt-\theta_c)^-\|_2^2 + \|\nabla(\bt-\theta_c)^-\|_2^2 \leq v(\bart(u(t),t))(\bt(u(t),t)-\theta_c)^-		\leq0,\qquad{\text{a.e. } t\in(0,T)}.
$$
An integration in time, using the hypothesis on $\bart_0$ leads to $\|(\bt-\theta_c)^-\|_{H^1}=0$ for a.e. $t\in(0,T)$, that is $\bart(s,t)\geq \theta_c$ everywhere in $\Omega$, for almost every $t\in[0,T)$. 
Let us now choose in \eqref{heat eq} $\psi=\bart^+:=\max{(0,\bart)}$, and notice that \eqref{cond v} together $\theta_c\leq \bt$ implies $v\leq 0$. Therefore,
$$
\dert \|\bt^+\|_2^2+ \|\nabla\bt^+\|_2^2\leq v(\bart(u(t),t))\bt^+(u(t),t)		\leq0,\qquad{\text{a.e. } t\in(0,T)}.
$$
This, together with the fact that $\|\bart^+_0\|_2=0$, proves the first statement. The second statement can be proved in an analogous way by first testing against $(\bt-\theta_c)^+$ and then against $\theta^-$.
\end{proof}

\begin{remark}
\label{RemMP}
\normalfont
The above proposition states that, under hypothesis \eqref{cond v}, a weak solution is bounded between $\theta_B$ and $\theta_T$. Therefore,
if our boundary condition $\theta_B$ is close to $\theta_T$, then the temperature solutions $\bart$ have small $L^\infty$ norm. Furthermore, this implies that the position of the interface is monotone, in the sense that if the boundary of the martensite sub-domain is expanding (resp. shrinking) then, as long as the boundary conditions do not change, the domain keeps expanding (resp. shrinking). 
\end{remark}

We can prove further regularity for $\theta$, as stated in the following
\begin{proposition}
\label{regolare}
Let $\bart_0\in H^1_0(\Omega)$, $0\leq \bart_0(s) \leq \theta_c$ (or $0\geq \bart_0(s) \geq \theta_c$) for each $s\in[0,1]$, $u_0\in\Omega$ and $v\in W^{1,\infty}_{loc}(\R)$ satisfy \eqref{cond v}. Then, for every $T>0$,
$$
\dot{\bart}\in L^2(0,T;L^2(\Omega)),\qquad\bart\in L^2(0,T;W^{1,\infty}(\Omega))\cap L^\infty(0,T;H^1_0(\Omega)).
$$
Furthermore, for a.e. $t\in(0,T)$ $\bart_s$ is continuous in $[0,u(t))\cap(u(t),L]$.
\end{proposition}
The proof of this result relies on
\begin{lemma}
\label{uniq par}
Let $v\in W^{1,\infty}_{loc}(\R)$. Let $(u,\bart_1)$ and $(u,\bart_2)$ be two weak solutions to Problem \eqref{pb 1d simple} sharing the same $u(t)\in H^1(0,T)$ and the same initial datum $\bart_0\in L^\infty(\Omega;[\min\{0,\theta_c\},\max\{0,\theta_c\}])$. Then, $\bart_1=\bart_2$.
\end{lemma}
\begin{proof}
Let $\te:=\bart_1-\bart_2$. Then, it satisfies
$$
\langle\te_t,\psi\rangle+\mint\te_s\psi_s\,\mathrm ds= (v(\bart_1(u(t),t))-v(\bart_2(u(t),t)))\psi(u(t)),
$$
for every $\psi\in H^1_0(0,L)$, a.e. $t\in(0,T)$. By testing this equation with $\psi=\te$ and using the fact that $v$ is Lipschitz we thus get
$$
\frac12\dert\|\te\|_2^2+\|\te_s\|_2^2 \leq c \te^2 (u(t),t).
$$
Now, \cite[Cor. 27]{Simon}, implies
$$
|\te (u(t),t)|\leq c \|\te\|_{W^{r,2}}(t), 
$$
with $r = \frac12+\eta$ and $\eta>0$ arbitrary. 
An interpolation inequality thus leads to
$$
\frac12\dert\|\te\|_2^2+\|\te_s\|_2^2 \leq c (\|\te\|_2^2+ \|\te_s\|_2^{2r}\|\te\|_2^{2(1-r)})\leq c \|\te\|_2^2+ \frac12 \|\te_s\|_2^2.
$$
The Gronwall lemma and the fact that $\te(0,s)=0$ for a.e. $s\in\Omega$ conclude the proof.
\end{proof}

\begin{proof}[Proof of Prop. \ref{regolare}]
Let us fix a solution $(\bart,u)$ among the possible solutions to Problem \eqref{pb 1d simple}. We drop the bar over $\bart$, and we put all the constants equal to one to simplify notation. 
Let us also define $\dep(s)=\delta_0^\eps(s-u(t))$ as the convolution of the Dirac measure centred in $u(t)$ with a smooth mollifier. $\dep$ is positive and converges $weakly^*$ in the sense of measures to $\delta_{u(t)}$ for every fixed $t$ as $\eps\to0$ (see \cite{AFP} for details). Furthermore, Fubini's Theorem yields $\|\dep\|_\mathcal M\leq c$, where $c$ is independent of $u(t)$ and $\eps$. For every fixed $\eps>0$ the approximate problem
\begin{align}
\label{eq approssimata bis}
\langle\dot \te_{\eps},\psi\rangle + \mint\te_{\eps,s}\psi_s\, \mathrm ds &= \mint v(\te_\eps(s,t))\dep(s)\psi(s)\,\mathrm ds,\quad &&\forall\psi\in H^1_0, \text{ a.e. $t\in(0,T)$}\\
\label{ci approssimata bis}
\te_\eps(s,0)&=\te_{0}(s),\quad &&\text{a.e. $s\in\Omega$}
\end{align}
\noindent
admits a solution $\te_\eps$. The following bound can be obtained thanks to Proposition \ref{MacPrin} by choosing $\psi=\te_\eps$ in \eqref{eq approssimata bis}
\beq
\label{es stab 1}
\|\te_{\eps}\|_{L^\infty(0,T;L^2)}+
\|\te_{\eps,s}\|_{L^2(0,T;L^2)}+
\|\dot\te_{\eps}\|_{L^2(0,T;H^{-1})}
\leq c_T,
\eeq
so that, by Lemma \ref{uniq par}, we may assume that
\begin{align}
\label{ten conv}
&\te_\eps\tow \te \quad \text{weakly  in } \Ltv\cap H^1(0,T;H^{-1}),\\
\label{strong conv}
&\te_\eps\to \te \quad \text{strongly  in }  L^2(0,T;C[0,L]).
\end{align}
We recall that here $\dot u(t) = v(\te(u(t),t))\chi_{\{u(t)\in(0,L)\}}$, as the equation for $u$ has not been approximated. Besides, by bootstrapping and using standard regularity for the heat equation (see e.g., \cite{Evans}) we also have $\te_{\eps,ss}\in L^2(0,T;H^2(\Omega))$ and $\te_{\eps,t}\in L^2(0,T;L^2(\Omega))$. Finally, an argument similar to that used to prove Proposition \ref{MacPrin} entails $\max_{s\in[0,L]}|\theta_\eps(s,t)|\leq \theta_c$ for a.e. $t\in(0,T).$ We can hence choose $\psi = \dot\te_{\eps}$ in \eqref{eq approssimata bis} and get
\beq
\begin{split}
\label{reg 01}
\|\dot\te_\eps\|_2^2+\frac12\dert\|\te_{\eps,s}\|_2^2&=\mint v(\te_\eps)\dep\dot\te_{\eps}\,\mathrm ds\\
&=\dert \mint V(\te_\eps)\dep\,\mathrm ds- \dot u \mint V(\te_\eps)(\dep)_{,s}\,\mathrm ds\\
&=\dert \mint V(\te_\eps)\dep\,\mathrm ds+ \dot u \mint \te_{\eps,s}v(\te_\eps)\dep\,\mathrm ds,
\end{split}
\eeq
where $V(r)=\int_0^r v(\tau)\,\mathrm d\tau$, and we integrated by parts in the last term. In addition, exploiting the additional regularity of $\te_\eps$, we notice that
\[
\begin{split}
|\te_{\eps,s}(x,t)| &\leq c\bigl( \|\te_{\eps,s}\|_{1} +\|\te_{\eps,ss}\|_{1}\bigr) \leq c\Bigl(\|\te_{\eps,s}\|_2 + \|\dot\te_{\eps}\|_{1} + \mint|v(\te_\eps(s,t))\dep(s)|\,\mathrm d s\Bigr).
\end{split}
\]
As $\dep$ is positive, and $v$ is Lipschitz,
$$
\mint |v(\te_\eps)\dep|\,\mathrm d s\leq \mint\dep |v(\te_\eps)|\,\mathrm d s \leq c(1+ \|\te_\eps\|_{C[0,L]})\mint \dep\,\mathrm d s\leq c(1+ \|\te_\eps\|_{C[0,L]}),
$$
which yields
$$
\max_{x\in[0,L]} |\te_{\eps,s}(x,t)|\leq c\bigl(\|\te_{\eps,s}\|_2 +  \|\dot\te_{\eps}\|_2  + 1+ \|\te_\eps\|_{C[0,L]} \bigr).
$$
Therefore,
\beq
\label{stimona}
\begin{split}
|\dot u|\Bigl|\mint \te_{\eps,s}v(\te_\eps)\dep\,\mathrm ds\Bigr|\leq c |\dot u|  \bigl(\|\te_{\eps,s}\|_2 +  \|\dot\te_{\eps}\|_2  + 1+ |\te_\eps|_{C[0,L]} \bigr)(1+\|\te_{\eps,s}\|_2)\\
\leq c+ c |\dot u|^2 (1+\|\te_{\eps,s}\|_2^2)+\frac 14 \bigl(\|\te_{\eps,s}\|_2^2 + \|\dot\te_{\eps}\|_2^2  +\|\te_{\eps,s}\|_2^2\bigr).
\end{split}
\eeq
Let us now consider a sequence of $\eps_j>0$ such that $\eps_j\to0$. We remark that \eqref{strong conv} implies $\te_{\eps_j}\to \te$ strongly in $L^2(0,T;C[0,L])$ and hence in $C[0,L]$ for almost every $t\in(0,T)$. Furthermore, as mentioned above, an argument as in the proof of Proposition \ref{MacPrin} tells us that $\|\te_{\eps_j}(\cdot,t)\|_{C[0,L]}\leq \theta_c$ for almost every $t$. 
We denote by $J_T$ the set of points in $(0,T)$ where the above uniform convergence holds and where the uniform bound for $\te_{\eps_j}(\cdot,t)$ holds for every ${\eps_j}$. We point out that this set is measurable and has full measure. 
After exploiting \eqref{stimona} in \eqref{reg 01}, and integrating the resulting inequality in time between $0$ and $t\in J_T$, we obtain
\beq
\begin{split}
\label{reg 02}
\tint \|\dot\te_{\eps_j}\|_2^2(\tau)\,\mathrm d\tau+\|\te_{\eps_j,s}\|_2^2(t)
\leq c+2\mint V(\te_{\eps_j}(s,t))\delta^{\eps_j}_{u(t)}(s)\,\mathrm ds - 2\mint V(\te_{\eps_j}(s,0))\delta^{\eps_j}_{u(0)}(s)\,\mathrm ds\\
 + 2\|\te_{\eps_j,s}\|_2^2(0)+c \tint|\dot u|^2 (1+\|\te_{\eps_j,s}\|_2^2)\,\mathrm d\tau+ \tint\|\te_{\eps_j,s}\|_2^2\,\mathrm d\tau.
\end{split}
\eeq
Now, by means of the regularity on the initial condition we easily deduce
$$
- 2\mint V(\te_{\eps_j}(s,0))\delta^{\eps_j}_{u(0)}(s)\,\mathrm ds + \|\te_{{\eps_j},s}\|_2^2(0)\leq c.
$$
and, as $V\in C^1(\R)$ and $t\in J_T$, 
\[
\begin{split}
2\mint V(\te_{\eps_j}(s,t)) \dep \,\mathrm ds &\leq 2 \|\dep\|_\mathcal M \|V(\te_{\eps_j})(\cdot,t)\|_{C[0,L]}
\leq c. 
\end{split}
\]
Finally, we know that \eqref{es stab 1} holds, and that, thanks to Proposition \ref{MacPrin} and Remark \ref{RemMP}, $\dot u\in L^\infty(0,T)$. Therefore, 
$$
\tint|\dot u|^2 (1+\|\te_{\eps_j,s}\|_2^2)\,\mathrm d\tau\leq c,
$$
and \eqref{reg 02} becomes
\[
\tint \|\dot\te_{\eps_j}(\tau)\|_2^2\,\mathrm{d}\tau+\|\te_{{\eps_j},s}\|_2^2(t)
\leq c,\qquad \text{a.e. $t\in (0,T)$},
\]
where $c$ depends neither on $t\in (0,T)$, nor on ${\eps_j}$. We can thus pass to the limit as $\eps\to 0$, and obtain
$$
\te\in L^2(0,T;L^2(\Omega)),\qquad \te_s\in L^\infty(0,T;L^2(\Omega)).
$$
By comparison, this also implies that $\te_{ss}+\dot u\delta_{u(t)}\in L^2(0,T;L^2(\Omega))$, that is $$\te_s+\dot u\chi_{\{s<u(t)\}}\in L^2(0,T;C[0,L]),$$ which yields
$$
\te\in L^2(0,T;W^{1,\infty}(\Omega)).
$$
\end{proof}
\begin{remark}
\rm{
The regularity proved in Proposition \ref{regolare} entails uniqueness thanks to an argument as that in \cite{Xien} for classical solutions. We remark that without having $\bart\in L^2(0,T;W^{1,\infty}(\Omega))$ uniqueness of the equation $\dot u = v(\bart(u(t),t))\chi_{\{ u(t)\in(0,L)\}}$ would be hopeless. Nonetheless, we point out that even if $\bart_0\notin H^1_0,$ the proof of Proposition \ref{regolare} entails that an instantaneous regularization takes place, and that in this case $\bart\in L^2(\rho,T;W^{1,\infty}(\Omega))$, for each $\rho>0$.
}
\end{remark}
As a corollary to the above regularity result we can prove that the austenite-martensite interface reaches the domain boundary in finite time, and that $\bart\to 0$ as $t\to\infty$. 
\begin{theorem}
Let $(\bart,u)$ be a weak solution in the sense of Definition \ref{Def ws}, let $u_0\in(0,L)$ and let $\bart_0\in H^1_0(\Omega)$ be such that $0\leq \bart_0(s)\leq \theta_c$ (or $0\geq \bart_0(s)\geq \theta_c$) for every $s\in\Omega$. Let also $v\in W_{loc}^{1,\infty}(\R)$ satisfy \eqref{cond v}. 
Then there exists $t^*>0$ such that $u(t^*)=L$ (resp. $u(t^*)=0$). Furthermore, $\|\bart\|_{2}(t)\to 0$ as $t\to\infty$.
\end{theorem}
\begin{proof}
We assume without loss of generality that $\theta_c>0$ and prove the theorem for the case where $0\leq \bart_0(s)\leq \theta_c$ for every $s\in\Omega$. The other case can be proved similarly. We remark that thanks to Proposition \ref{MacPrin} and \eqref{cond v} we have $0\leq \dot u(t)\leq \max_{s\in[0,\theta_c]}v(s)$ for almost every $t\geq 0$, so that $u$ is non-decreasing and bounded. Therefore $\lim_{t\to\infty} u(t)$ exists. Suppose now that there exists no $t^*<\infty$ such that $u(t^*)=L$. Then, for every $\eps>0$ there exists $t_0$ such that $$
\int_{t_0}^\infty \dot u(\tau)\,\mathrm{d}\tau \leq \eps. $$ 
Let us now test \eqref{heat eq} with $\bart$ to obtain
$$
\frac12\dert\|\bart\|^2_2(t) +\|\bart_s\|^2_2(t) = \dot{u}(t)\bart(u(t),t),\qquad\text{a.e. $t\geq 0$}.
$$
Thanks to the Sobolev inequality $\|\psi\|_{C[0,L]}\leq c\|\psi_s\|_2$, which holds for every $\psi\in H^1_0(\Omega)$, by Young's inequality we thus deduce
$$
\frac{d}{d\tau}\|\bart\|^2_2(t_0+\tau) +\|\bart_s\|^2_2(t_0+\tau) \leq c|\dot{u}|^2(t_0+\tau) \leq \hat c|\dot{u}|(t_0+\tau) ,\qquad\text{a.e. $\tau\geq 0$},
$$
for some constants $c,\hat c>0$ and where we made use of $|\dot u|\leq c$. Using once again Sobolev embeddings we have
$$
\frac{d}{d\tau} y(\tau) + c_Py(\tau) \leq \hat c|\dot{u}|(t_0+\tau) ,\qquad\text{a.e. $\tau\geq 0$},
$$
for some $c_P>0$ and where we defined $y(\tau):=\|\bart\|^2_2(t_0+\tau)$. Using the comparison principle for ordinary differential equations we get that
$$
y(\tau)\leq y(0)e^{-c_P \tau} +e^{-c_P \tau}\int_0^\tau e^{c_P \hat\tau} \hat c|\dot{u}|(t_0+\hat \tau)\,\mathrm{d}\hat\tau. 
$$
As $\|\bart\|_2$ is a continuous function of time and as $\|\bart\|_2\leq L^{\frac12}\theta_c$ for a.e. $t\geq 0$, we have $y(0)\leq L\theta_c^2$. Let us now choose $\tau$ such that $L\theta_c^2e^{-c_0 \tau}\leq \eps$, and such that $\bart(u(t_0+\tau),t_0+\tau) \geq \frac{3\theta_c}{4}$. This is possible because, as $u\in(0,L)$ for each $t\geq0$, $\dot u=0$ if and only if $\bart(u(t),t) = \theta_c$, and because $v(s)>0$ if $0\leq s<\theta_c$ (cf. \eqref{cond v}). Indeed
$$
\eps \geq \int_{t_0}^{\infty}\dot{u}(\tau)\,\mathrm{d}\tau
\geq \min\Bigl\{v(s)\colon s\in\Bigl[0,\frac{3\theta_c}{4}\Bigr]\Bigr\}\mathscr{L}\Bigl(\Bigl\{t\in[t_0,\infty)\colon \bart(u(t),t)\in \Bigl[0,\frac{3\theta_c}{4}\Bigr] \Bigr\}\Bigr),
$$
that is $\bart(u(t),t)\notin \Bigl[\frac{3\theta_c}{4},\theta_c\Bigr]$ only for a finite amount of time. Thus, as we are assuming that the interface position stays in $(0,L)$ for infinite time, such a value of $\tau$ exists. 
We obtain
\beq
\label{epssmall}
y(\tau)\leq \eps + \hat c\int_0^\tau |\dot{u}|(t_0+\hat \tau)\,\mathrm{d}\hat\tau \leq (1+\hat c) \eps. 
\eeq
On the other hand, there exists $\tilde{c}>0$ such that for almost every $t\geq 0$ we have
\beq
\label{holderiano}
|\bart(u(t),t)-\bart(x(t),t)|\leq \tilde c\|\bart_s\|_2(t) |u(t)-x(t)|^{\frac12}. 
\eeq
Proposition \ref{regolare} also implies that $\bart\in L^\infty(0,T;H^1_0(\Omega))$, and we can therefore assume without loss of generality that $\|\bart_s\|_2(t_0+\tau)\leq \hat \alpha$, for some constant $\hat\alpha>0$. 
Let us now consider
$$\mathcal{D}_\tau := \Bigl\{s\in (0,u(t_0+\tau))\colon \bart(s,t_0+\tau)=\frac{\theta_c}4 \Bigr\}.$$
This set is closed, and non-empty given the fact $\bart(\cdot,t_0+\tau)$ is continuous, that $\bart(0,t_0+\tau)=0$ and that $\bart(u(t_0+\tau),t_0+\tau)>\frac{3\theta_c}{4}$. Let us hence choose $x(\tau)=\max{\mathcal D_\tau}$ so that \eqref{holderiano} becomes
$$
\frac{\theta_c}2 = |\bart(u(t_0+\tau),t_0+\tau)-\bart(x(\tau),t_0+\tau)|\leq \tilde c\hat\alpha |u(t_0+\tau)-x(\tau)|^{\frac12}. 
$$
This yields
$$
\|\bart\|^2_2(t_0+\tau)= \mint |\bart(s,t_0+\tau)|^2\,\mathrm{d}s\geq \frac{\theta^2_c}{16}|u(t_0+\tau)-x(\tau)|\geq \frac{\theta^4_c}{64\tilde c\hat\alpha}.
$$
Combining this inequality with \eqref{epssmall}, by choosing $\eps$ small enough, we are led to a contradiction. Therefore, for every $t\geq t^*$ \eqref{heat eq} simplifies to a standard homogeneous heat equation with initial datum in $H^1_0(\Omega)$, for which convergence to the only equilibrium, namely $0$, is a well-known result (see e.g., \cite{Evans}). 
\end{proof}

\textbf{Acknowledgements:} {This work was supported by the Engineering and Physical Sciences Research Council [EP/L015811/1]. The author would like to thank John Ball for his helpful suggestions and feedback which greatly improved this work. The author would like to acknowledge the anonymous reviewer for improving this paper with his comments.}
\bibliographystyle{plain}
\footnotesize
\bibliography{biblio}

\begin{thebibliography}{10}

\bibitem{AberKnowles}
R.~Abeyaratne and J.K. Knowles.
\newblock On the driving traction acting on a surface of strain discontinuity
  in a continuum.
\newblock {\em Journal of the Mechanics and Physics of Solids}, 38(3):345--360,
  1990.

\bibitem{AberKnowles2}
R.~Abeyaratne and J.K. Knowles.
\newblock A continuum model of a thermoelastic solid capable of undergoing
  phase transitions.
\newblock {\em Journal of the Mechanics and Physics of Solids}, 41(3):541--571,
  1993.

\bibitem{AFP}
L.~Ambrosio, N.~Fusco, and D.~Pallara.
\newblock {\em Functions of bounded variation and free discontinuity problems}.
\newblock Oxford Mathematical Monographs. The Clarendon Press, Oxford
  University Press, New York, 2000.

\bibitem{PF3}
A.~Artemev, Y.~Jin, and A.G. Khachaturyan.
\newblock Three-dimensional phase field model of proper martensitic
  transformation.
\newblock {\em Acta Materialia}, 49(7):1165 -- 1177, 2001.

\bibitem{BallYM}
J.M. Ball.
\newblock A version of the fundamental theorem for {Y}oung measures.
\newblock In {\em P{DE}s and continuum models of phase transitions ({N}ice,
  1988)}, volume 344 of {\em Lecture Notes in Phys.}, pages 207--215. Springer,
  Berlin, 1989.

\bibitem{BallOpen}
J.M. Ball.
\newblock Some open problems in elasticity.
\newblock In {\em Geometry, mechanics, and dynamics}, pages 3--59. Springer,
  New York, 2002.

\bibitem{BallSengul}
J.M. Ball and Y.~\c~Seng\"ul.
\newblock Quasistatic nonlinear viscoelasticity and gradient flows.
\newblock {\em J. Dynam. Differential Equations}, 27(3-4):405--442, 2015.

\bibitem{BallChuJames}
J.M. Ball, C.~Chu, and R.D. James.
\newblock Hysteresis during stress-induced variant rearrangement.
\newblock {\em Le Journal de Physique IV}, 5(C8):C8--245, 1995.

\bibitem{BallPego}
J.M. Ball, P.J. Holmes, R.D. James, R.L. Pego, and P.J. Swart.
\newblock On the dynamics of fine structure.
\newblock {\em J. Nonlinear Sci.}, 1(1):17--70, 1991.

\bibitem{BallJames1}
J.M. Ball and R.D. James.
\newblock Fine phase mixtures as minimizers of energy.
\newblock {\em Arch. Rational Mech. Anal.}, 100(1):13--52, 1987.

\bibitem{BallJames2}
J.M. Ball and R.D. James.
\newblock Proposed experimental tests of a theory of fine microstructure and
  the two-well problem.
\newblock {\em Phil. Trans. R. Soc. Lond. A}, 338(1650):389--450, 1992.

\bibitem{BallKoumatos}
J.M. Ball and K.~Koumatos.
\newblock An investigation of non-planar austenite-martensite interfaces.
\newblock {\em Math. Models Methods Appl. Sci.}, 24(10):1937--1956, 2014.

\bibitem{BallKoumatosQC}
J.M. Ball and K.~Koumatos.
\newblock Quasiconvexity at the boundary and the nucleation of austenite.
\newblock {\em Arch. Ration. Mech. Anal.}, 219(1):89--157, 2016.

\bibitem{BattSme}
K.~Bhattacharya.
\newblock Self-accommodation in martensite.
\newblock {\em Arch. Rational Mech. Anal.}, 120(3):201--244, 1992.

\bibitem{Batt}
K.~Bhattacharya.
\newblock {\em Microstructure of martensite}.
\newblock Oxford Series on Materials Modelling. Oxford University Press,
  Oxford, 2003.
\newblock Why it forms and how it gives rise to the shape-memory effect.

\bibitem{Boyer}
F.~Boyer and P.~Fabrie.
\newblock {\em Mathematical tools for the study of the incompressible
  {N}avier-{S}tokes equations and related models}, volume 183 of {\em Applied
  Mathematical Sciences}.
\newblock Springer, New York, 2013.

\bibitem{JamesHyst}
X.~Chen, V.~Srivastava, V.~Dabade, and R.D. James.
\newblock Study of the {\it cofactor conditions}: conditions of
  supercompatibility between phases.
\newblock {\em J. Mech. Phys. Solids}, 61(12):2566--2587, 2013.

\bibitem{XC}
X.~Chen, N.~Tamura, A.~MacDowell, and R.D. James.
\newblock In-situ characterization of highly reversible phase transformation by
  synchrotron x-ray laue microdiffraction.
\newblock {\em Applied Physics Letters}, 108(21):211902, 2016.

\bibitem{Chluba}
C.~Chluba, W.~Ge, R.L. de~Miranda, J.~Strobel, L.~Kienle, E.~Quandt, and
  M.~Wuttig.
\newblock Ultralow-fatigue shape memory alloy films.
\newblock {\em Science}, 348(6238):1004--1007, 2015.

\bibitem{All}
S.~Conti, G.~Dolzmann, and S.~M\"uller.
\newblock The div-curl lemma for sequences whose divergence and curl are
  compact in {$W^{-1,1}$}.
\newblock {\em C. R. Math. Acad. Sci. Paris}, 349(3-4):175--178, 2011.

\bibitem{ContiSchweizer}
S.~Conti and B.~Schweizer.
\newblock Rigidity and gamma convergence for solid-solid phase transitions with
  {SO}(2) invariance.
\newblock {\em Comm. Pure Appl. Math.}, 59(6):830--868, 2006.

\bibitem{FDP2}
F.~Della~Porta.
\newblock Analysis of a moving mask approximation for martensitic
  transformations.
\newblock {\em In review}.

\bibitem{FDP4}
F.~Della~Porta.
\newblock Interfacial energy as a selection mechanism for minimising gradient
  {Y}oung measures in a one-dimensional model problem.
\newblock {\em Accepted}.

\bibitem{FDP3}
F.~Della~Porta.
\newblock On the cofactor conditions and further conditions of
  supercompatibility between phases.
\newblock {\em Journal of the Mechanics and Physics of Solids}, 122:27 -- 53,
  2019.

\bibitem{MullerDolzman}
G.~Dolzmann and S.~M\"uller.
\newblock Microstructures with finite surface energy: the two-well problem.
\newblock {\em Arch. Rational Mech. Anal.}, 132(2):101--141, 1995.

\bibitem{Evans}
L.C. Evans.
\newblock {\em Partial differential equations}, volume~19 of {\em Graduate
  Studies in Mathematics}.
\newblock American Mathematical Society, Providence, RI, second edition, 2010.

\bibitem{Flanders}
H.~Flanders.
\newblock Differentiation under the integral sign.
\newblock {\em Amer. Math. Monthly}, 80:615--627; correction, ibid. 81 (1974),
  145, 1973.

\bibitem{PF2}
F.E. Hildebrand and C.~Miehe.
\newblock A phase field model for the formation and evolution of martensitic
  laminate microstructure at finite strains.
\newblock {\em Philosophical Magazine}, 92(34):4250--4290, 2012.

\bibitem{AberKim}
S.J. Kim and R.~Abeyaratne.
\newblock On the effect of the heat generated during a stress-induced
  thermoelastic phase transformation.
\newblock {\em Contin. Mech. Thermodyn.}, 7(3):311--332, 1995.

\bibitem{KohnMuller}
R.V. Kohn and S.~M\"uller.
\newblock Surface energy and microstructure in coherent phase transitions.
\newblock {\em Comm. Pure Appl. Math.}, 47(4):405--435, 1994.

\bibitem{PF1}
M.~Mamivand, M.A. Zaeem, and H.~El~Kadiri.
\newblock A review on phase field modeling of martensitic phase transformation.
\newblock {\em Computational Materials Science}, 77:304 -- 311, 2013.

\bibitem{Muller}
S.~M\"uller.
\newblock Variational models for microstructure and phase transitions.
\newblock In {\em Calculus of variations and geometric evolution problems
  ({C}etraro, 1996)}, volume 1713 of {\em Lecture Notes in Math.}, pages
  85--210. Springer, Berlin, 1999.

\bibitem{MullerSverak}
S.~M\"uller and V.~\v{S}ver\'{a}k.
\newblock Convex integration with constraints and applications to phase
  transitions and partial differential equations.
\newblock {\em J. Eur. Math. Soc. (JEMS)}, 1(4):393--422, 1999.

\bibitem{Pedregal}
P.~Pedregal.
\newblock {\em Parametrized measures and variational principles}, volume~30 of
  {\em Progress in Nonlinear Differential Equations and their Applications}.
\newblock Birkh\"auser Verlag, Basel, 1997.

\bibitem{Pego}
R.L. Pego.
\newblock Phase transitions in one-dimensional nonlinear viscoelasticity:
  admissibility and stability.
\newblock {\em Arch. Rational Mech. Anal.}, 97(4):353--394, 1987.

\bibitem{Res}
J.G. Re\v{s}etnjak.
\newblock Liouville's conformal mapping theorem under minimal regularity
  hypotheses.
\newblock {\em Sibirsk. Mat. \v{Z}.}, 8:835--840, 1967.

\bibitem{Sengul_thesis}
Y.~Sengul.
\newblock {\em Well-posedness of dynamics of microstructure in solids}.
\newblock PhD thesis, University of Oxford, 2010.

\bibitem{Simon}
J.~Simon.
\newblock Sobolev, {B}esov and {N}ikoskii fractional spaces: imbeddings and
  comparisons for vector valued spaces on an interval.
\newblock {\em Ann. Mat. Pura Appl. (4)}, 157:117--148, 1990.

\bibitem{JamesNew}
Y.~Song, X.~Chen, V.~Dabade, T.W. Shield, and R.D. James.
\newblock Enhanced reversibility and unusual microstructure of a
  phase-transforming material.
\newblock {\em Nature}, 502(7469):85, 2013.

\bibitem{Truesdell}
C.~Truesdell.
\newblock {\em Rational thermodynamics}.
\newblock Springer-Verlag, New York, second edition, 1984.
\newblock With an appendix by C. C. Wang, With additional appendices by 23
  contributors.

\bibitem{Vis}
A.~Visintin.
\newblock Stefan problem with a kinetic condition at the free boundary.
\newblock {\em Ann. Mat. Pura Appl. (4)}, 146:97--122, 1987.

\bibitem{VisLib}
A.~Visintin.
\newblock {\em Models of phase transitions}, volume~28 of {\em Progress in
  Nonlinear Differential Equations and their Applications}.
\newblock Birkh\"auser Boston, Inc., Boston, MA, 1996.

\bibitem{Xien}
W.Q. Xie.
\newblock The {S}tefan problem with a kinetic condition at the free boundary.
\newblock {\em SIAM J. Math. Anal.}, 21(2):362--373, 1990.

\bibitem{JamesMuller}
Z.~Zhang, R.D. James, and S.~M{\"u}ller.
\newblock Energy barriers and hysteresis in martensitic phase transformations.
\newblock {\em Acta Materialia}, 57(15):4332--4352, 2009.

\end{thebibliography}

\end{document}